\documentclass[reqno, 12pt]{article}

\pdfoutput=1

\usepackage{enumerate}
\usepackage{latexsym}
\usepackage[centertags]{amsmath}
\usepackage{amsfonts}
\usepackage{amssymb}
\usepackage{amsthm}
\usepackage{newlfont}
\usepackage{graphics}
\usepackage{color}
\usepackage{float}
\usepackage{diagbox}
\usepackage{tocloft}
\usepackage{titlesec}
\usepackage{booktabs}
\usepackage[pagebackref]{hyperref}
\usepackage[linesnumbered,ruled,vlined]{algorithm2e}
\usepackage{url}
\usepackage{hyperref}
\usepackage{longtable}
\usepackage{rotating}
\usepackage{multirow}
\usepackage{extarrows}
\usepackage[sort,compress,numbers]{natbib}
\usepackage[utf8]{inputenc}

\newtheorem{thm}{Theorem}[section]
\newtheorem{cor}[thm]{Corollary}
\newtheorem{lem}[thm]{Lemma}
\newtheorem{prop}[thm]{Proposition}

\theoremstyle{mydefinition}
\newtheorem{dfn}[thm]{Definition}
\theoremstyle{myremark}

\newtheorem{exa}[thm]{Example}

\newtheorem{prob}[thm]{Open Problem}

\allowdisplaybreaks[4]

\SetKwInput{KwInput}{Input}                
\SetKwInput{KwOutput}{Output}              

\def\Z{\mathbb{Z}}

\renewcommand{\P}{{\mathbb{P}}}

\title{Explicit Construction of Polytopes whose Ehrhart Polynomials Realize any Given Sign Pattern}

\author{Feihu Liu$^{\color{blue} \dag}$, Sihao Tao$^{\color{blue} \S}$, and Guoce Xin$^{\color{blue} \P}$
\\[2mm]
{\small $^{\color{blue} \dag, \S, \P}$ School of Mathematical Sciences,}\\[-0.8ex]
{\small Capital Normal University, Beijing, 100048, P.R.~China}\\
{\small {\color{blue} $^\dag$} Email address: liufeihu7476@163.com}\\
{\small {\color{blue} $^\S$} Email address: sihao\_tao@cnu.edu.cn}\\
{\small {\color{blue} $^\P$} Email address: guoce\_xin@163.com}
}

\date{\today}

\begin{document}

\maketitle

\begin{abstract}
In Ehrhart theory, the well-known sign pattern problem asks: given a positive integer $d\geq 3$ and integers $1 \leq i_1 < \cdots < i_k \leq d-2$, does there exist a $d$-dimensional integral polytope $\mathcal{P}$ such that in its Ehrhart polynomial $i(\mathcal{P}, t)$ the coefficients of $t^{i_1}, \ldots, t^{i_k}$ are negative, while all remaining coefficients are positive?
This problem was proposed by Hibi, Higashitani, Tsuchiya, and Yoshida.
In this paper, we first construct a class of simplices $\mathcal{S}_d(m)$ whose Ehrhart polynomial has leading coefficient $m$ and all other coefficients fixed positive constants.
Then, using the Cartesian product of $\mathcal{S}_d(m)$ and the Reeve tetrahedron, we obtain the first complete solution to the sign pattern problem.
Finally, while attacking the sign pattern problem, we discovered a fast algorithm for computing the $h^*$-polynomial of a class of simplices $\Delta(0,q)$.
This algorithm is crucial for constructing the simplices $\mathcal{S}_d(m)$.
\end{abstract}

\noindent
\begin{small}
\end{small}

\noindent
\begin{small}
\end{small}

\tableofcontents

\section{Introduction}

\subsection{Basic concepts}

An  \emph{integral convex polytope} (also called a \emph{lattice polytope}) is the convex hull of finitely many integer points in $\mathbb{R}^d$. For a finite set $\{v_1,v_2,\ldots,v_n\} \subset \mathbb{Z}^d$,
$$
\mathcal{P} = \mathrm{conv}\{v_1, \dots, v_n\} = \left\{\sum_{i=1}^n \lambda_i v_i \mid \lambda_i \geq 0,\ \sum_{i=1}^n \lambda_i = 1 \right\}.
$$
Its \emph{dimension} $\dim \mathcal{P}$ is the dimension of its affine hull. Its \emph{Ehrhart function}, defined as $i(\mathcal{P},t) = |t\mathcal{P}\cap \mathbb{Z}^d|$ for positive integers $t$, counts the number of integer points in the $t$-th dilation $t\mathcal{P} = \{ t\alpha \mid \alpha \in \mathcal{P} \}$. We always assume $\mathcal{P}$ is full-dimensional (i.e., $\dim \mathcal{P}=d$).

Ehrhart's famous theorem \cite{Ehrhart62} says that $i(\mathcal{P},t)$ is a polynomial in $t$ of degree $d$, now known as the \emph{Ehrhart polynomial} of $\mathcal{P}$.
It is well-known \cite[Corollary 3.20; Theorem 5.6]{BeckRobins} that the leading coefficient of $i(\mathcal{P},t)$ equals the volume of $\mathcal{P}$, the second highest coefficient of $i(\mathcal{P},t)$ equals half of the boundary volume of $\mathcal{P}$, and the constant term is $1$.
Therefore, these three coefficients are always positive.

The remaining coefficients, however, are more intricate and lack a straightforward geometric interpretation \cite{McMullen77}.
We refer to the coefficients of $t,t^2,\ldots,t^{d-2}$ in $i(\mathcal{P},t)$ as the \emph{middle Ehrhart coefficients} of $\mathcal{P}$.

The generating function
$$\mathrm{Ehr}(\mathcal{P},x)=1+\sum_{t\geq 1} i(\mathcal{P},t)x^t$$
is called the \emph{Ehrhart series} of $\mathcal{P}$.
It can be expressed in the form
$$\mathrm{Ehr}(\mathcal{P},x)=\frac{h_{\mathcal{P}}^{*}(x)}{(1-x)^{d+1}}=\frac{h_d^{*}x^d + h_{d-1}^{*}x^{d-1}+ \cdots+ h_1^*x+ h_0^*}{(1-x)^{d+1}},$$
where $\dim\mathcal{P}=d$ and $h^{*}(x)$ is a polynomial in $x$ of degree at most $d$ (see \cite[Chapter 3.5]{BeckRobins}).
The polynomial $h_{\mathcal{P}}^*(x)$ is commonly known as the \emph{$h^*$-polynomial} of $\mathcal{P}$.

Stanley \cite{Stanleyh-polynomial} (or \cite[Theorem 3.12]{BeckRobins}) shows that $h_d^*,h_{d-1}^*,\ldots,h_0^*$ are nonnegative integers.
Furthermore, the polynomial $h^*(x)$ satisfies \cite[Lemma 3.13; Corollary 3.16; Exercise 4.9]{BeckRobins}
\begin{align}\label{formula-h^*-prop}
h_0^* = 1,\quad h_1^* = \#(\mathcal P \cap \mathbb{Z}^d) - (d+1),\quad\text{ and }\quad h_d^* = \#(\mathrm{int}(\mathcal{P}) \cap \mathbb{Z}^d),
\end{align}
where $\mathrm{int}(\mathcal{P})$ denotes the interior of $\mathcal{P}$.
Clearly the Ehrhart polynomial is given by
\begin{align}\label{formual-ehrbyh*}
i(\mathcal{P},t) = \sum_{i=0}^{d} h^*_i \binom{t + d - i}{d}.
\end{align}

An integral convex polytope $\mathcal{P}$ is said to be \emph{Ehrhart positive} (or have \emph{Ehrhart positivity}) if all the coefficients of $i(\mathcal{P},t)$ are non-negative.
A well-known non-Ehrhart positive example is Reeve's tetrahedron.
\begin{dfn}{\em \cite[Example 3.22]{BeckRobins}}\label{ReeveTetrahedron}
Let $\mathcal{T}_m \subseteq \mathbb{R}^3$ be the tetrahedron with vertices $(0,0,0)$, $(1,0,0)$, $(0,1,0)$, and $(1,1,m)$, where $m\in\mathbb{N}^+$. Its Ehrhart polynomial is given by
\begin{align}\label{equ_tetrahedron}
i(\mathcal{T}_m,t)=\frac{m}{6}t^3+t^2+\frac{12-m}{6}t+1.
\end{align}
This tetrahedron is known as the \emph{Reeve tetrahedron}.
\end{dfn}

The study of Ehrhart positivity has attracted considerable attention in recent years.
Numerous families of polytopes have been established as Ehrhart positive, including lattice path matroids \cite{Ferroni-Morales2026}
(which encompass the hypersimplices~\cite{Hypersimplices}, minimal matroids~\cite{MinimalMatroids}, Schubert matroids~\cite{Fan-Li}, Catalan matroids~\cite{Chen-Li-Yao}), rank-two matroids~\cite{Rank-Two-Matroids}, cross-polytopes~\cite[Section~2]{FuLiu19}, the $y$-families of generalized permutohedra~\cite{PostnikovIMRN} (which encompass the Pitman--Stanley polytopes~\cite{Stanley-Pitman}), and cyclic polytopes~\cite{LiuFuA-M}.

Conversely, many polytopes have been shown to be non-Ehrhart positive, including order polytopes~\cite{LiuTsuchiya19,LiuXinZhang}, matroid polytopes~\cite{MatroidsPolytope}, and smooth polytopes~\cite{Smoothpolytopes}.
For a comprehensive introduction to Ehrhart positivity, we refer the reader to Liu's survey \cite{FuLiu19}.

Inspired by the study on Reeve tetrahedra and Ehrhart positivity, Hibi, Higashitani, Tsuchiya, and Yoshida \cite{Hibi-Higashitani-Tsuchiya-Yoshida} investigated possible sign patterns of coefficients in Ehrhart polynomials and proposed the following question:
\begin{prob}{\em \cite[Question 4.1]{Hibi-Higashitani-Tsuchiya-Yoshida}; \cite[Question 4.2.1]{FuLiu19}; \cite[Conjecture 4.3]{Ferroni23}}\label{QuestionHHTY}
Given a positive integer $d\geq 3$ and integers $1 \leq i_1 < \cdots < i_q \leq d-2$, does there exist a $d$-dimensional integral polytope $\mathcal{P}$ such that the coefficients of $t^{i_1}, \ldots, t^{i_q}$ in $i(\mathcal{P}, t)$ are negative, while the remaining coefficients are positive?
\end{prob}

Hibi, Higashitani, Tsuchiya, and Yoshida provided partial answers to this question \cite[Theorem 1.1; Theorem 1.2]{Hibi-Higashitani-Tsuchiya-Yoshida}.  Hibi et al also proved that Open Problem \ref{QuestionHHTY} holds for all $3\leq d\leq 6$.
As noted in Liu's review \cite{FuLiu19}, Tsuchiya proved that any sign pattern with at most $3$ negatives is possible among the middle Ehrhart coefficients.

In \cite{LiuTaoXin}, the authors proved the following theorem by establishing several embedding theorems. The basic tool is the Cartesian product of dilated polytopes, particularly the Reeve tetrahedron and the polytope to be called the Eulerian simplex of dimension 2.
\begin{thm}{\em \cite{LiuTaoXin}}\label{thm-eq-Only-Verified}
Let $d \geq 6$ and assume that Problem~\ref{QuestionHHTY} holds for all polytopes of dimensions $1$ through $d-1$. Then, to verify Problem~\ref{QuestionHHTY} for all $d$-dimensional polytopes, it suffices to check that for any integers $d_i \geq 2$ ($1\leq i\leq k-1$) and $d_k\geq 3$ satisfying $d=d_1+d_2+\cdots+d_k+3$, there exists a $d$-dimensional integral polytope $\mathcal{P}$ such that
$$\operatorname{Sgn}(\mathcal{P}) = (-1, \underbrace{+1, -1, \cdots, -1}_{d_1}, \underbrace{+1, -1, \ldots, -1}_{d_2}, \ldots, \underbrace{+1, -1, \ldots, -1}_{d_k}).$$
\end{thm}

The theorem allows us to verify Open Problem \ref{QuestionHHTY} in a small number of cases. Indeed, we settled the $d\leq 9$ cases in \cite{LiuTaoXin}.
Nevertheless, this problem remains widely open.

\subsection{Main results}
To solve this open problem, we need to construct a polytope with any given sign patterns as stated in Theorem \ref{thm-eq-Only-Verified}.
These cases seem hard to attack, but we solved them by adding a new block, called Eulerian simplices of dimension $d\geq 1$,
because their $h^*$-polynomials are similar to the Eulerian polynomials $A_d(x)$ (see Proposition \ref{Generating Function for Eulerian Polynomials}).
\begin{dfn}
For every positive integer $m$ (i.e., $m\in \mathbb{Z}^+$),
a parameterized family $\mathcal{R}_d(m)$ of $d$-dimensional integral simplices in $\mathbb{R}^d$ is called an \emph{Eulerian simplex family of dimension $d$} if the $h^*$-polynomial of $\mathcal{R}_d(m)$ admits a decomposition of the form
$$h^*_{\mathcal{R}_d(m)}(x) = m A_d(x) + T(x),$$
where $T(x)$ is a polynomial independent of $m$ with $T(1)=0$, $\deg T(x)\leq d$.
\end{dfn}

As the main tool of this paper, we construct a class of simplices, denoted by $\mathcal{S}_d(m)$, which belongs to the Eulerian simplex family.
\begin{thm}\label{Theorem-Introdu-Sk}
For every positive integer $m$ and $d\geq 1$, the following family $\mathcal{S}_d(m)$ of polytopes fall into the Eulerian simplex family $\mathcal{R}_d(m)$ of dimension $d$:
\begin{align*}
\mathcal{S}_d(m)=
\begin{cases}
\ell_{m}=[0,m]=\{\alpha\in \mathbb{R} \mid 0\leq \alpha\leq m\} & \text{ if }\quad d=1, \\
\mathrm{conv}\{\mathbf{0}^d, \mathbf{e}_1^d, \mathbf{e}_2^d, \ldots, \mathbf{e}_{d-1}^d, (q_1(d), q_2(d), \ldots, q_{d-1}(d),d!\cdot m)\} & \text{ if }\quad d\geq 2,
\end{cases}
\end{align*}
where $\mathbf{e}_i^d$ denote the $i$-th unit coordinate vector of $\mathbb{R}^d$; $\mathbf{0}^d$ denote the origin of $\mathbb{R}^d$; and
$$q_i(d)=\frac{-d!}{i!+(i-1)!}\quad \text{for} \quad 1\leq i\leq d-1.$$
\end{thm}

Moreover, we compute the closed formula of the Ehrhart polynomial of $\mathcal{S}_d(m)$.
\begin{thm}\label{Theorem-Introd-Eulor-Ehrhar}
Let $m\geq 1$. For every positive integer $d$, the Ehrhart polynomials of $\mathcal{S}_d(m)$ are given by
\begin{align}\label{Equation-Eula-Sd}
i(\mathcal{S}_d(m),t)=mt^d+\sum_{i=0}^{d-1}\binom{d}{i} t^i.
\end{align}
\end{thm}
Observe that the middle Ehrhart coefficients and the second highest coefficient of a polytope are sufficiently small compared to the leading coefficient. This greatly facilitates the construction of polytopes for certain sign patterns.

Our main result is a resolution of Open Problem \ref{QuestionHHTY} by combining our constructed simplex family $\mathcal{S}_d(m)$, the Reeve tetrahedron, and Cartesian products of polytopes.
\begin{thm}\label{Theorem-Introdu-Sign-parrt}
For any dimension $d\geq 3$ and every possible sign pattern of the coefficients in the Ehrhart polynomial,
we can construct an integral polytope that realizes it.
\end{thm}

The construction of the above Eulerian simplices $\mathcal{S}_d(m)$ is a little unexpected. We begin with a general family of simplices (see Section \ref{Section-Four-Detla})
$$\Delta(0,q):=\mathrm{conv}\left\{\mathbf{0}^d,\mathbf{e}_1^d,\ldots,\mathbf{e}_{d-1}^d, n\mathbf{e}_d^d+\sum_{i=1}^{d-1}q_i\mathbf{e}_i \right\}.$$
These simplices can be treated as a generalization of the Reeve tetrahedron. Such simplices have been considered in \cite{HNFsimplex}, where the $q_i$ are restricted to be positive and less than $n$.

We develop a fast algorithm for computing the $h^*$-polynomial of $\Delta(0,q)$. This enhances the search of $\mathcal{S}_d(m)$.
As a \emph{by-product}, $\mathcal{S}_d(m)$ is a simplex family of a counterintuitive phenomenon:
These are simplices whose volume tends to infinity but with constant boundary volume as $m$ increases. See Section \ref{sub-Section-RIP}.

The paper is organized as follows.
In Section \ref{Section-Two-Sign}, we prove Theorem \ref{Theorem-Introdu-Sign-parrt} under the assumption of Theorem \ref{Theorem-Introd-Eulor-Ehrhar},
whose proof occupies the remaining sections. In other words, we resolve the sign pattern problem.
Section \ref{Section-Three-Fundam} provides some results on the fundamental parallelepiped that are used in the following sections.
In Section \ref{Section-Four-Detla}, we introduce the simplices $\Delta(0,q)$, and derive a formula that characterizes their Ehrhart polynomials. Furthermore, we provide a fast algorithm for computing the $h^*$-polynomial of $\Delta(0,q)$.
In Section \ref{Section-Five}, we concentrate on the properties of the $h^*$-polynomials for a special class of simplices in $\Delta(0,q)$. They contain many known simplices as special cases.
Section \ref{Section-Six} is mainly concerned with the construction of the Eulerian simplex family $\mathcal{S}_d(m)$.
Furthermore, we also obtain the closed formula for the Ehrhart polynomial of $\mathcal{S}_d(m)$.
Section \ref{Section-Seven-CR} is a concluding remark.

\section{Sign pattern problem}\label{Section-Two-Sign}

In this section, we will make full use of $\mathcal{S}_d(m)$ to study the sign pattern problem for polytopes.
Before this, we need to introduce some preliminary notions and results.

\subsection{Preliminaries}\label{Section-Two-Sign-1}

Let $\mathcal{P}\subseteq\mathbb{R}^{n}$ and $ \mathcal{Q}\subseteq\mathbb{R}^{m}$ be convex polytopes. Then \emph{Cartesian product} of $\mathcal{P}$ and $\mathcal{Q}$ is the convex polytope defined by
$$\mathcal{P}\times \mathcal{Q}\;=\;\left\{(x,y)\in\mathbb{R}^{n+m}\bigm| x\in \mathcal{P},\; y\in \mathcal{Q}\right\}.$$

If $\mathcal{P}$ and $\mathcal{Q}$ are integral polytopes of dimensions $d_1$ and $d_2$, respectively, then $\mathcal{P}\times \mathcal{Q}$ is an integral polytope of dimension $d_1+d_2$.
Moreover, its Ehrhart polynomial \cite[Lemma 2.3]{Hibi-Higashitani-Tsuchiya-Yoshida} is given by
\begin{align*}
 i(\mathcal{P}\times \mathcal{Q},t)= i(\mathcal{P},t)\cdot  i(\mathcal{Q},t).
\end{align*}

In the proof of our main result, we also need the following obvious formula.
For any positive integer $r$, we have
$$i(r\mathcal{P},t) = i(\mathcal{P},rt).$$

Define the sign function $\mathrm{sgn}(n)$ for an integer $n$ by $\mathrm{sgn}(n) = -1$ if $n < 0$, $\mathrm{sgn}(n) = 0$ if $n = 0$, and $\mathrm{sgn}(n) = +1$ if $n > 0$.

\begin{dfn}
Let $\mathcal{P}$ be an integral convex polytope of dimension $d$ with $i(\mathcal{P},t)=1+\sum_{i=1}^{d}c_it^i$. Then we define the \emph{sign vector} of $\mathcal{P}$ by
$$\operatorname{Sgn}(\mathcal{P}) = (\operatorname{sgn}(c_{d-2}), \ldots, \operatorname{sgn}(c_{1})),$$
where $\operatorname{sgn}$ denotes the sign function.
\end{dfn}

Let $\{d_i\}_{i=1}^k$ be a sequence of integers with $d_i \ge 2$. Note that here $d_k\geq 2$, which differs from the condition $d_k\geq 3$ in Theorem \ref{thm-eq-Only-Verified}.
We consider the Cartesian product of Eulerian simplices and the Reeve tetrahedron as follows:
$$\mathcal{P}:=\left(\prod_{i=1}^k r_i \mathcal{S}_{d_i}(m_i) \right)\times r_0\mathcal{T}_{m_0},$$
where $r_i=r^{\alpha_i}$ and $m_i=r^{\beta_i}$ for $0\leq i\leq k$; $\alpha_i$ and $\beta_i$ are positive rational numbers.

For any sequences of positive rational numbers $(\alpha_i)_{i=0}^k$ and $(\beta_i)_{i=0}^k$, as $r \to +\infty$ along the sequence of perfect $L$-th powers (where $L$ is the least common multiple of the denominators of all $\alpha_i$ and $\beta_i$), both $r_i$ and $m_i$ are positive integers. Consequently, $\mathcal{P}$ is an \textbf{integral} polytope. In what follows, whenever we write $r \to +\infty$, it is implicitly assumed that $r$ approaches infinity along this sequence.

According to Equations \eqref{equ_tetrahedron} and \eqref{Equation-Eula-Sd} (as $r \to +\infty$), we obtain
\begin{align*}
i(r_0\mathcal{T}_{m_0}, t)&=\frac{1}{6}r^{3\alpha_0+\beta_0}t^3+r^{2\alpha_0}t^2+\left(-\frac{1}{6} r^{\alpha_0+\beta_0}+2r^{\alpha_0} \right)t+1,
\\ i(r_i\mathcal{S}_d(m_i),t)&=r^{\beta_i +d_i \alpha_i} t^{d_i} +\sum_{\ell=0}^{d_i-1}\binom{d}{i} r^{\ell \alpha_i} t^{\ell}.
\end{align*}

Since we wish to examine the sign pattern of the coefficients in $i(\mathcal{P},t)$ as $r\to \infty$, we need only consider the highest power of $r$ in the above expression together with its sign. Therefore, we can drop the irrelevant constant coefficients.
For convenience, we define $P_j(t)$ for $j=0,\dots,k$ as the corresponding Ehrhart polynomial:
\begin{align*}
P_0(t) &= i(r_0 \mathcal{T}_{m_0},t) = r^{\beta_0+3\alpha_0} t^3+r^{2\alpha_0}t^2-r^{\alpha_0+\beta_0}t+1,\\
P_i(t) &= i(r_i \mathcal{S}_{d_i}(m_i),t) = r^{\beta_i +d_i \alpha_i} t^{d_i} +\sum_{\ell=0}^{d_i-1} r^{\ell \alpha_i} t^{\ell}, \quad \text{for}\quad 1\leq i\leq k.
\end{align*}
In the next subsection, we study the expansion of the product $P(t) =i(\mathcal P,t)= \left(\prod_{i=1}^k P_i(t)\right)\cdot P_0(t)$.

\subsection{Proof of main result}

Given an arbitrary sequence of integers $d_i \ge 2$, we construct positive rational parameters $\alpha_i, \beta_i$ such that the coefficients of the expanded polynomial product strictly follow a prescribed sign pattern as the scaling parameter $r$ tends to infinity. Specifically, we demonstrate that by employing a delicately weighted assignment, the asymptotic signs can be completely determined by the marginal increments of a greedy decomposition over the degree sequence.

\begin{thm}\label{Thm-main-Sgn-PP0}
For any $d_i \geq 2$ ($1\leq i\leq k$), let $D=d_1+d_2+\cdots+d_k$.
There exist positive rational sequences $(\alpha_i)_{i=0}^k$ and $(\beta_i)_{i=0}^k$ such that
$$P(t)=\left(\prod_{i=1}^k P_i(t) \right)\cdot P_0(t)$$
is a polynomial of degree $D+3$ in $t$ and, as $r\to +\infty$, the sign pattern of $\mathcal P$ satisfies
\begin{align*}
\operatorname{Sgn}(\mathcal P) = (-1, \underbrace{+1, -1, \cdots, -1}_{d_1}, \underbrace{+1, -1, \ldots, -1}_{d_2}, \ldots, \underbrace{+1, -1, \ldots, -1}_{d_k}),
\end{align*}
In other words, the coefficients of $t^{D+3}$, $t^{D+2}$, $t^{d_j+d_2\cdots +d_k}$ for $1\leq j\leq k$, and the constant term in $P(t)$ are positive, while the remaining coefficients are negative.
\end{thm}
\begin{proof}
The polynomial $P(t)$ has degree $D+3$, and we can write its expansion as
$$P(t) = \sum_{h=0}^{D+3} C_h(r) t^h,$$
Where $r \to +\infty$, and $\alpha_i, \beta_i > 0$ are constants to be determined.
Define the suffix sums $E_j = \sum_{i=j}^k d_i$ for $1 \le j \le k$.
We denote the target index set of positive coefficients as
$$ \mathbf{T} = \{0, D+2, D+3\} \cup \{E_j \mid 1 \le j \le k\}. $$
We need to prove that for any degree sequence $\{d_i\}$, there exists a universal construction of parameters ensuring
$$C_h(r) > 0\quad \text{for}\quad h \in \mathbf{T}\quad \text{and}\quad C_h(r) < 0\quad \text{for}\quad h \notin \mathbf{T}$$
for sufficiently large $r$.

To determine the sign of $C_h(r)$ as $r \to +\infty$, it suffices to track the maximum exponent of $r$ in the expansion, which we define as the \emph{$r$-weight}.
For $1 \le i \le k$, the polynomial $P_i(t)$ contributes a weight $w_i(\ell_i)$ for the term $t^{\ell_i}$ as follow:
\begin{align*}
w_i(\ell_i) =
\begin{cases}
\ell_i \alpha_i, & \text{if}\quad 0 \le \ell_i < d_i \\
 d_i \alpha_i + \beta_i, &\text{if}\quad \ell_i = d_i.
\end{cases}
\end{align*}
Let $X_i = w_i(d_i) = d_i\alpha_i + \beta_i$.
For the polynomial $P_0(t)$, the weights and corresponding signs $s_0(l_0)$ are
\begin{itemize}
    \item $w_0(0) = 0$, $s_0(0) = 1$;
    \item $w_0(1) = \alpha_0 + \beta_0$, $s_0(1) = -1$;
    \item $w_0(2) = 2\alpha_0$, $s_0(2) = 1$;
    \item $w_0(3) = 3\alpha_0 + \beta_0$, $s_0(3) = 1$.
\end{itemize}
Define the maximum weight obtainable from the product $\prod_{i=1}^k P_i(t)$ for a given degree $0\le x \le D$ as
\begin{align*}
W(x) = \max \left\{ \sum_{i=1}^k w_i(\ell_i) \;\middle|\; \sum_{i=1}^k \ell_i = x, \; 0 \le \ell_i \le d_i \right\}.
\end{align*}
The overall maximal weight for degree $h$ is given by maximizing $w_0(\ell_0) + W(h - \ell_0)$ over allowable $\ell_0 \in \{0, 1, 2, 3\}$.

We explicitly construct the parameter sequences $(\alpha_i)_{i=0}^k$ and $(\beta_i)_{i=0}^k$ ensuring the global validity of a greedy decomposition of $x$.
Choose $\epsilon$ such that
\begin{align*}
0< \epsilon \ll \prod_{j=1}^{k-1} \left( 1 - \frac{1}{d_j} \right)< 1\quad (\text{Since}\quad  d_j \ge 2).
\end{align*}
Let $\alpha_0 = \frac{\epsilon}{3}$ and $\beta_0 =1-\frac{\epsilon}{3}$.
Consequently, the weights are
$$w_0(0)=0,\quad w_1(0)=1, \quad w_0(2) =\frac{2 \epsilon}{3}, \quad w_0(3) =1+\frac{2 \epsilon}{3}.$$
Define the parameters $\alpha_i$ and $\beta_i$ for $1 \le i \le k$ recursively:
\begin{align}
& \alpha_1 = \epsilon, \label{eq-alpha-one}\\
& \beta_i = 1 + \epsilon - \alpha_i, \quad (\text{i.e., } \alpha_i + \beta_i = 1 + \epsilon), \label{eq-beta} \\
& \alpha_{i+1} = \alpha_i + \frac{\beta_i}{d_i}, \quad (\text{for } 1\leq i\leq k-1). \label{eq-alpha_rec}
\end{align}

\begin{lem} \label{lem-param-bounds}
The sequence $(\alpha_i)_{i=0}^{k}$ satisfies $0<\alpha_0< \epsilon= \alpha_1 < \alpha_2 < \dots < \alpha_k < 1$.
Furthermore, $1>\beta_i > \epsilon > 0$ for all $1\leq i\leq k$.
\end{lem}
\begin{proof}
Substituting \eqref{eq-beta} into \eqref{eq-alpha_rec}, we have $\alpha_{i+1} = \alpha_i + \frac{1+\epsilon-\alpha_i}{d_i}$, which is equivalent to
\begin{align*}
1+ \epsilon - \alpha_{i+1} = (1 + \epsilon - \alpha_i) \left( 1 - \frac{1}{d_i} \right).
\end{align*}
Iterating this from $i=1$ yields
\begin{align*}
1 + \epsilon - \alpha_k = (1 + \epsilon - \alpha_1) \prod_{j=1}^{k-1} \left( 1 - \frac{1}{d_j} \right).
\end{align*}
Rearranging gives $1 - \alpha_k = \prod_{j=1}^{k-1} \left( 1 - \frac{1}{d_j} \right) - \epsilon$.
By the initial choice of $\epsilon$, $1 - \alpha_k > 0$, implying $\alpha_k < 1$.
Since $\beta_i = 1 + \epsilon - \alpha_i > \epsilon > 0$, all parameters $\beta_i$ are positive and less than $1$.
\end{proof}

Lemma \ref{lem-param-bounds} implies that the sequence $(\alpha_i)$ is strictly monotonically increasing, $(\beta_i)$ is strictly monotonically decreasing, and the maximum weight for the $i$-th component satisfies
$$w_i(d_i) = d_i\alpha_i + \beta_i = d_i\alpha_{i+1}.$$

We define a decomposition $S^* = (\ell_1^*, \dots, \ell_k^*)$ for a degree $x$ to be \emph{greedy} if there exists an $m\in \{1,\ldots,k\}$ such that
$$\ell_i^* = d_i\quad\text{for}\quad i > m,\quad 0 \le \ell_m^* < d_m,\quad \ell_i^* = 0\quad \text{for}\quad i < m, \quad\text{and}\quad \sum_{i=1}^{k}\ell_i^*=x.$$
Let $S = (\ell_1, \dots, \ell_k)$ be an arbitrary valid decomposition satisfying $\sum_{i=1}^k \ell_i = x$.
The total weight of a decomposition $S$ is given by $W(S) = \sum_{i=1}^k w_i(\ell_i)$. The global maximum weight for a fixed sum $x$ is denoted as $W(x) = \max_S W(S)$.
\begin{lem}\label{lem-greedy}
For any integer $x$ such that $0 \le x \le D$, the greedy decomposition $S^*$ realizes the maximum weight, that is, $W(S^*) = W(x) \ge W(S)$ for any valid decomposition $S$.
\end{lem}
\begin{proof}
We now define
\begin{align*}
V = \sum_{i=m+1}^k (d_i - \ell_i) \geq 0.
\end{align*}
Since both decompositions sum to the same total $x$, we have
\begin{align*}
\sum_{i=m+1}^k d_i + \ell_m^* + \sum_{i=1}^{m-1} 0 = \sum_{i=m+1}^k \ell_i + \ell_m + \sum_{i=1}^{m-1} \ell_i.
\end{align*}
Rearranging the terms yields
\begin{align}\label{eq-conservation}
V = \sum_{i=m+1}^k (d_i - \ell_i) = \sum_{i=1}^{m-1} \ell_i + (\ell_m - \ell_m^*).
\end{align}

We consider the three cases as follows.

\textbf{Case I: ($i > m$)}
For each $i > m$, let $\delta_i = d_i - \ell_i \ge 0$.
We have
\begin{align*}
w_i(d_i) - w_i(\ell_i) =
\begin{cases}
0, & \text{if}\quad \delta_i = 0 \\
(d_i\alpha_i + \beta_i) - \ell_i\alpha_i = \delta_i\alpha_i + \beta_i, &\text{if}\quad \delta_i \ge 1.
\end{cases}
\end{align*}
Since $i \ge m+1$, the monotonic increasing property guarantees $\alpha_i \ge \alpha_{m+1} > \alpha_m$.
Recall that $\alpha_i + \beta_i =1+ \epsilon$ for all $i\geq 1$.
Then we obtain
\begin{align*}
\delta_i\alpha_i + \beta_i &= \delta_i\alpha_m + \delta_i(\alpha_i - \alpha_m) + \beta_i\\
&\geq \delta_i\alpha_m+ 1 \cdot (\alpha_i - \alpha_m) + \beta_i \\
&=\delta_i\alpha_m+ 1+\epsilon - \alpha_m
= \delta_i\alpha_m+\beta_m.
\end{align*}
Thus, for any $\delta_i \ge 0$, we have $w_i(d_i) - w_i(\ell_i) \ge \delta_i\alpha_m$. Furthermore, if $V > 0$, there exists at least one $j > m$ such that $\delta_j \ge 1$, which strictly contributes an additional $\beta_m$.
We now employ the characteristic function $\chi$: $\chi(\text{true})=1$ and $\chi(\text{false})=0$.
Therefore, we get
\begin{align}\label{eq-high-order}
\sum_{i=m+1}^k (w_i(d_i) - w_i(\ell_i)) \ge V\alpha_m + \chi(V > 0)\beta_m.
\end{align}

\textbf{Case II: ($i < m$)}
In this case, $\ell_i^* = 0$, so we seek to upper bound $w_i(\ell_i)$.
If $\ell_i < d_i$, then $w_i(\ell_i) = \ell_i\alpha_i$. Since $i \le m-1$, we have $\alpha_i \le \alpha_{m-1} < \alpha_m$. Then we have $w_i(\ell_i) \le \ell_i\alpha_m$.
If $\ell_i = d_i$, then the weight is $w_i(d_i) = d_i\alpha_i + \beta_i = d_i\alpha_{i+1}$. Because $i < m$  (i.e., $i+1 \le m$), it follows that $\alpha_{i+1} \le \alpha_m$. Hence, $w_i(d_i) \le d_i\alpha_m$.
Therefore, we obtain
\begin{align}\label{eq-low-order}
\sum_{i=1}^{m-1} (w_i(\ell_i^*) - w_i(\ell_i)) = -\sum_{i=1}^{m-1} w_i(\ell_i) \ge -\sum_{i=1}^{m-1} \ell_i\alpha_m.
\end{align}

\textbf{Case III: ($i = m$)} The condition $\ell_m^* < d_m$ imply $w_m(\ell_m^*) = \ell_m^*\alpha_m$.
For the arbitrary $\ell_m$, we have
$$w_m(\ell_m) = \ell_m\alpha_m + \chi(\ell_m = d_m)\beta_m.$$
Therefore, we get
\begin{align} \label{eq-pivot}
w_m(\ell_m^*) - w_m(\ell_m) = (\ell_m^* - \ell_m)\alpha_m - \chi(\ell_m = d_m)\beta_m.
\end{align}

Equations \eqref{eq-high-order}, \eqref{eq-low-order}, and \eqref{eq-pivot} establish a lower bound as follow
\begin{align*}
W(S^*) - W(S) &\ge (V\alpha_m + \chi(V > 0)\beta_m) + ((\ell_m^* - \ell_m)\alpha_m - \chi(\ell_m = d_m)\beta_m) - \sum_{i=1}^{m-1} \ell_i\alpha_m
\\ & =\left( V - (\ell_m - \ell_m^*) - \sum_{i=1}^{m-1} \ell_i \right) \alpha_m + (\chi(V > 0) - \chi(\ell_m = d_m))\beta_m.
\end{align*}
By Equation \eqref{eq-conservation}, the inequality become
\begin{align*}
W(S^*) - W(S) \ge (\chi(V > 0) - \chi(\ell_m = d_m))\beta_m.
\end{align*}

It remains to prove that $\chi(V > 0) - \chi(\ell_m = d_m) \ge 0$. It suffices to show the case $\chi(\ell_m = d_m)=1$ and $\chi(V > 0)=0$ cannot happen.
If not, then $V=0$ and $\ell_m = d_m$. Substituting $\ell_m = d_m$ into Equation \eqref{eq-conservation} yields the contradiction:
$$V = \sum_{i=1}^{m-1} \ell_i + (d_m - \ell_m^*)>0.$$

Furthermore, we obtain $W(S^*) \ge W(S)$.
This completes the proof that the greedy strategy guarantees the maximum weight, i.e., Lemma \ref{lem-greedy}.
\end{proof}

Define the marginal increment $\Delta(x) = W(x) - W(x-1)$.
\begin{lem}
For $x = E_j$ with $1\leq j \leq k$, we have $\Delta(E_j) = 1 + \epsilon$.
For $E_{j+1}<x <E_{j}$, we have $\Delta(x) = \alpha_j\le \alpha_k < 1$.
\end{lem}
\begin{proof}
By Lemma \ref{lem-greedy}, when $x = E_j$, we have $\Delta(E_j)=\alpha_j+\beta_j = 1 + \epsilon$.
For $E_{j+1}<x <E_{j}$, the greedy decomposition has $0 < \ell_j^* < d_j$.
Decreasing the degree by 1 removes one unit from $\ell_j^*$, reducing the weight strictly by $\alpha_j$. Lemma \ref{lem-param-bounds} guarantees $\alpha_j \le \alpha_k < 1$.
\end{proof}

We define the weights combining $P_0(t)$ and $\prod P_i(t)$ as $W_{\ell_0}(h) = w_0(\ell_0) + W(h - \ell_0)$, where $0\leq \ell_0\leq 3$.
It is clear that when $r\to +\infty$, the sign (positive or negative) of the coefficient $C_h(r)$ of polynomial $P(t)$ is determined by the $\max_{\ell_0} W_{\ell_0}(h)$.
If the corresponding $\ell_0$ of $\max_{\ell_0} W_{\ell_0}(h)$ equals $1$, then $C_h(r)<0$ as $r\to +\infty$; otherwise $C_h(r)>0$ as $r\to +\infty$.

We now consider the sign of $C_h(r)$ according to the following four cases.

\textbf{Case 1: $h = E_j$ with $1\leq j \leq k$.}
According to the
$$ W_0(E_j) - W_1(E_j) = W(E_j) -(\alpha_0+\beta_0 + W(E_j-1)) = \Delta(E_j) - 1 = \epsilon > 0,$$
we have $C_h(r)>0$ as $r\to +\infty$.

\textbf{Case 2: $h \in\{D+3,D+2,0\}$.}
\begin{itemize}
    \item When $h=D+3$, only $\ell_0=3$ with $x=D$ is valid. The coefficient $C_h(r)$ is positive as $r\to +\infty$.
    \item When $h=D+2$, the valid $\ell_0 \in \{2, 3\}$. The coefficient $C_h(r)$ is positive as $r\to +\infty$.
    \item When $h=0$, the constant term $C_0(r)$ is always $1$.
\end{itemize}

For $h \notin \mathbf{T}$, we must prove $W_1(h)$ strictly dominates all legal positive components $W_0(h)$, $W_2(h)$, $W_3(h)$.

\textbf{Case 3: $1 \le h \le D$ and $h \neq E_j$.}
\begin{itemize}
    \item $\ell_0=1$ versus $\ell_0=0$: We have
          $$W_1(h) - W_0(h) = \alpha_0+\beta_0 + W(h-1) - W(h) = 1 - \Delta(h) \ge 1 - \alpha_k > 0.$$
    \item $\ell_0=1$ versus $\ell_0=2$: We have
          $$W_1(h) - W_2(h) = \alpha_0+\beta_0 + W(h-1) - (2\alpha_0 + W(h-2)) = \beta_0-\alpha_0 + \Delta(h-1).$$
          Since $\Delta(h-1) \ge \alpha_1 > 2\alpha_0$, we get $W_1(h) - W_2(h) > 0$.
    \item $\ell_0=1$ versus $\ell_0=3$: We have
          $$W_1(h) - W_3(h) = \alpha_0+\beta_0 + W(h-1) - (3\alpha_0+\beta_0 + W(h-3)) = \Delta(h-1) + \Delta(h-2) - 2\alpha_0.$$ Since $\Delta(h-1)\geq \alpha_1$ and $\Delta(h-2)\geq \alpha_1$, we obtain $W_1(h) - W_3(h)\geq 2(\alpha_1-\alpha_0)>0$.
\end{itemize}
Thus, in this case, $C_h(r)<0$ holds as $r\to +\infty$.

\textbf{Case 4: $h = D+1$.} The negative component is $W_1(D+1) = \alpha_0+\beta_0 + W(D)$.
\begin{itemize}
    \item $\ell_0=1$ versus $\ell_0=2$: We obtain
          $$W_1(h) - W_2(h) = \alpha_0+\beta_0 + W(D) - (2\alpha_0 + W(D-1)) =1-\frac{2\epsilon}{3}+1+\epsilon > 0.$$
    \item $\ell_0=1$ versus $\ell_0=3$: We have
          \begin{align*}
          W_1(h) - W_3(h) = 1 + W(D) - (3\alpha_0+\beta_0 + W(D-2)) &= -2\alpha_0+W(D)-W(D-2)
          \\ &=-2\alpha_0+\beta_1+2\alpha_1>0.
          \end{align*}
\end{itemize}
Hence, in this case, we have $C_{D+1}(r)<0$ as $r\to +\infty$.

This completes the proof of Theorem \ref{Thm-main-Sgn-PP0}.
\end{proof}

Clearly, through Theorems~\ref{thm-eq-Only-Verified} and~\ref{Thm-main-Sgn-PP0}, we have proved Open Problem \ref{QuestionHHTY}, namely:
Given a positive integer $d\geq 3$ and integers $1 \leq i_1 < \cdots < i_q \leq d-2$, there exists a $d$-dimensional integral polytope $\mathcal{P}$ such that the coefficients of $t^{i_1}, \ldots, t^{i_q}$ in $i(\mathcal{P}, t)$ are negative, while the remaining coefficients are positive. Consequently, Theorem \ref{Theorem-Introdu-Sign-parrt} holds.

\section{Fundamental parallelepiped}\label{Section-Three-Fundam}

In solving the above sign pattern problem, an important tool is the discovery of the simplex family $\mathcal{S}_d(m)$. Before describing this discovery, we first introduce some necessary concepts concerning simplices and polytopes.

A $d$-dimensional convex polytope with exactly $d+1$ vertices is called a $d$-dimensional \emph{simplex}.
For example, the $3$-dimensional simplices are the tetrahedra.
The \emph{cone over $\mathcal{P}$} is defined as
$$\mathrm{cone}(\mathcal{P}) = \{ (\lambda, x) \in \mathbb{R}_{\geq 0} \times \mathbb{R}^d : x \in \lambda \mathcal P\}.$$

Let $\mathcal{P}$ be a $d$-dimensional integral simplex with vertices $v_0, v_1, \ldots, v_d$. The \emph{fundamental parallelepiped} of $\mathrm{cone}(\mathcal{P})$ is
$$\Pi(\mathrm{cone}(\mathcal{P})):=\{ \lambda_0(1,v_0)+\lambda_1(1,v_1)+\cdots +\lambda_d(1,v_d) : \lambda_i\in [0,1)\},$$
and its Ehrhart series \cite[Lemma 3.10]{BeckRobins} is given by
\begin{align}\label{formual--cone(P)hstar}
 \mathrm{Ehr}(\mathcal{P}, x) = \frac{\sum_{\omega \in \Pi(\mathrm{cone}(\mathcal{P}))\cap \mathbb{Z}^{d+1}} x^{\omega_0}}{(1-x)^{d+1}},
\end{align}
where $\omega_0$ is the first coordinate of $\omega=(\omega_0,\omega_1,\ldots,\omega_d)$.

The concept of a fundamental parallelepiped can be viewed within a broader algebraic framework.
For a nonsingular matrix $A \in \mathbb{Z}^{d \times d}$, the \emph{lattice generated by $A$} is defined as
$\mathcal{L}(A) =\{Ax : x\in\mathbb{Z}^d\}$.
The \emph{fundamental parallelepiped} of the lattice $\mathcal{L}(A)$ is defined as
$$\Pi(A) = \{A\xi : \xi \in [0,1)^d\}.$$

\begin{prop}{\em \cite[Theorem 2.5]{barvinok2025course}}\label{Properties of Lattices and Fundamental Parallelepipeds}
Let $A \in \mathbb{Z}^{d \times d}$ be a nonsingular matrix and $\mathcal{L}(A)$ the lattice generated by $A$. Then we have
\begin{enumerate}
    \item The quotient group $\mathbb{Z}^d/\mathcal{L}(A)$ is finite.
    \item The number of elements in the quotient group $\mathbb{Z}^d/\mathcal{L}(A)$ equals $|\det(A)|$.
    \item The set $\Pi(A)\cap\mathbb{Z}^d$  forms a complete set of representatives for the quotient group $\mathbb{Z}^d/\mathcal{L}(A)$.
\end{enumerate}
\end{prop}

The following lemma appears in \cite{Barvinok1993}, but without a detailed proof. Here we provide one.

\begin{lem}{\em \cite{Barvinok1993}}\label{lem-SNF-of-Pi(cone)}
Let $A, B \in \mathbb{Z}^{d \times d}$ be two nonsingular matrices such that $B=UAV$ for some unimodular matrices $U, V \in \mathbb{Z}^{d \times d}$.
Let
\begin{align*}
\Pi(A)\cap \mathbb{Z}^d = \{A\xi \cap \mathbb{Z}^d : \xi \in [0,1)^d\} \quad \text{and}\quad
\Pi(B)\cap \mathbb{Z}^d = \{B\xi \cap \mathbb{Z}^d : \xi \in [0,1)^d\}.
\end{align*}
Then
\begin{align*}
\Pi(A) \cap \mathbb{Z}^d= \{A(VB^{-1}b - \lfloor VB^{-1}b \rfloor) : b \in \Pi(B)\cap \mathbb{Z}^d\}.
\end{align*}
\end{lem}
\begin{proof}
Let $G_1 = \mathbb{Z}^d / \mathcal{L}(B)$ and $G_2 = \mathbb{Z}^d / \mathcal{L}(A)$. By Proposition~\ref{Properties of Lattices and Fundamental Parallelepipeds}, $\Pi(B) \cap \mathbb{Z}^d$ and $\Pi(A) \cap \mathbb{Z}^d$ are complete sets of representatives for $G_1$ and $G_2$, respectively. Moreover, $|G_1| = |\det(B)| = |\det(A)| = |G_2|$.

Define a map $\phi: G_1 \to G_2$ by
\[
\phi(b + \mathcal{L}(B)) = A V B^{-1} b + \mathcal{L}(A).
\]

\begin{enumerate}
  \item $\phi$ is a well-defined group homomorphism.
If $b_1 - b_2 \in \mathcal{L}(B)$, then $b_1 - b_2 = B z$ for some $z \in \mathbb{Z}^d$. Thus
\[
A V B^{-1}(b_1 - b_2) = A V z \in \mathcal{L}(A),
\]
so $\phi(b_1 + \mathcal{L}(B)) = \phi(b_2 + \mathcal{L}(B))$. The homomorphism property is immediate from linearity.
  \item  $\phi$ is a group isomorphism.
Define $\psi: G_2 \to G_1$ by
\[
\psi(a + \mathcal{L}(A)) = B V^{-1} A^{-1} a + \mathcal{L}(B).
\]
It is straightforward to verify that $\psi \circ \phi = \text{id}_{G_1}$ and $\phi \circ \psi = \text{id}_{G_2}$.
\end{enumerate}
Finally, for any $b \in \Pi(B) \cap \mathbb{Z}^d$, write
\[
V B^{-1} b = \left\lfloor V B^{-1} b \right\rfloor + \xi, \quad \xi \in [0,1)^d.
\]
Then
\[
\phi(b + \mathcal{L}(B)) = A \xi + \mathcal{L}(A),
\]
which means $A\xi \in \Pi(A) \cap \mathbb{Z}^d$. This completes the proof.
\end{proof}

\section{Simplex family: $\Delta(0,q)$}\label{Section-Four-Detla}

\subsection{Formula for the $h^*$-Polynomials}

We now introduce a special class of simplices, denoted by $\Delta(0,q)$, which exhibits numerous elegant properties. Most notably, this class of simplices is central to the study of Ehrhart sign pattern problems.
\begin{dfn}
Let $q=(q_1,\ldots,q_{d-1},n) \in \mathbb{Z}^d$ with $n\neq 0$, we define the simplex
\begin{align}\label{define--Delat(0,q)}
\Delta(0,q):=\mathrm{conv}\left\{\mathbf{0},\mathbf{e}_1,\ldots,\mathbf{e}_{d-1}, n\mathbf{e}_d+\sum_{i=1}^{d-1}q_i\mathbf{e}_i \right\},
\end{align}
where $\mathbf{e}_i$ denote the $i$-th unit coordinate vector of $\mathbb{R}^d$ for $1\leq i\leq d$; and $\mathbf{0}$ denote the origin of $\mathbb{R}^d$.
\end{dfn}

Similar families of simplices have been studied in \cite{HNFsimplex} and \cite{Braun-Davis-Solus}.
However, the simplices here are slightly different from theirs.

\begin{thm}\label{thm-Delta}
The $h^*$-polynomial of $\Delta(0,q)$ is given by
\begin{align*}
h_{\Delta(0,q)}^*(x)=\sum_{j=0}^{n-1}x^{\left\lceil{\frac{q_1 j}{n}}\right\rceil+\cdots+\left\lceil{\frac{q_d j}{n}}\right\rceil}
\quad \text{ with }\quad q_d=1-\sum_{i=1}^{d-1}q_i.
\end{align*}
\end{thm}
\begin{proof}
We construct a nonsingular matrix $A$ and a unimodular matrix $V$ as follows:
\begin{align*}
A = \begin{pmatrix}
1 & 1 & \cdots & 1 & 1 \\
0 & 1 & \cdots & 0 & q_1 \\
\vdots & \vdots & \ddots & \vdots & \vdots \\
0 & 0 & \cdots & 1 & q_{d-1} \\
0 & 0 & \cdots & 0 & n
\end{pmatrix}, \quad
V = \begin{pmatrix}
1 & -1 & \cdots & -1 & -q_d \\
0 & 1 & \cdots & 0 & -q_1 \\
\vdots & \vdots & \ddots & \vdots & \vdots \\
0 & 0 & \cdots & 1 & -q_{d-1} \\
0 & 0 & \cdots & 0 & 1
\end{pmatrix}.
\end{align*}
Direct computation yields $AV = \mathrm{diag}\{1, \ldots, 1, n\}$. By Lemma~\ref{lem-SNF-of-Pi(cone)} (here, $B=AV$), we obtain
\begin{align*}
&\Pi(\mathrm{cone}(\Delta(0,q)))\cap\mathbb{Z}^{d+1}=\Pi(A)\cap\mathbb{Z}^{d+1}
\\=& \left\{ A(\xi_j - \lfloor \xi_j \rfloor) :
\xi_j = V\mathrm{diag}\left\{0, \ldots, 0, \frac{j}{n}\right\}
= \frac{j}{n} \begin{pmatrix}
-q_d \\ -q_1 \\ \vdots \\ -q_{d-1} \\ 1
\end{pmatrix}, \ j = 0, \ldots, n-1 \right\}.
\end{align*}
We observe that the first coordinates of the vectors $(A\xi_j)$ and $(-A\lfloor \xi_j \rfloor)$, denoted by $(A\xi_j)_0$ and $(-A\lfloor \xi_j \rfloor)_0$ respectively, satisfy $(A\xi_j)_0 = 0$ and
\begin{align*}
(-A\lfloor \xi_j \rfloor)_0
&= -\left(\left\lfloor \frac{-q_d j}{n} \right\rfloor +\left\lfloor \frac{-q_1 j}{n}\right\rfloor+\cdots
+\left\lfloor \frac{-q_{d-1} j}{n}\right\rfloor+\left\lfloor \frac{j}{n}\right\rfloor \right)
\\&= -\left( \left\lfloor \frac{j}{n} \right\rfloor + \sum_{i=1}^{d} \left\lfloor \frac{-q_i j}{n} \right\rfloor \right)
= \sum_{i=1}^{d} \left\lceil \frac{q_i j}{n} \right\rceil.
\end{align*}
The conclusion now follows from Equation~\eqref{formual--cone(P)hstar}.
\end{proof}

We now consider the $h^*$-polynomials for three special cases of $\Delta(0,q)$.

\begin{cor}
Let $a,s\in\mathbb{Z}_{\geq 0}$.
Let $d = 2k - 1$ for $k \geq 2$, and consider the $d$-dimensional simplex defined as
$$\mathcal{R}^{\mathrm{odd}}_{s,k,a}:= \mathrm{conv}\{0, \mathbf e_1, \ldots, \mathbf e_{d-1}, (\underbrace{a, \ldots, a}_{k-1} , \underbrace{-a, \ldots, -a}_{k-1}, s + 1)\} \subset \mathbb{R}^d,$$
The $h^*$-polynomial of this simplex is given by
\begin{align*}
h^*_{\mathcal{R}^{\mathrm{odd}}_{s,k,a}}(x) =(s+1-b)x^k+(b-1)x+1, \quad \text{where}\quad b = \gcd(a, s+1).
\end{align*}
\end{cor}
\begin{proof}
By Theorem~\ref{thm-Delta}, we have
\begin{align*}
h^*_{\mathcal{R}^{\mathrm{odd}}_{s,k,a}}(x) &= \sum_{i=0}^{s}x^{\left\lceil{\frac{ai}{s+1}}\right\rceil+\cdots+\left\lceil{\frac{ ai}{s+1}}\right\rceil+\left\lceil{\frac{-ai}{s+1}}\right\rceil+\cdots+\left\lceil{\frac{-a i}{s+1}}\right\rceil+\left\lceil{\frac{(1-(k-1)a+(k-1)a) i}{s+1}}\right\rceil}
\\&=1+\sum_{i=1}^{s} {x^{(k-1)\left(\left\lceil{\frac{ai}{s+1}}\right\rceil-\left\lfloor{\frac{ai}{s+1}}\right\rfloor\right)} x^{\left\lceil{\frac{i}{s+1}}\right\rceil}}
\\&=1+\sum_{i=1 \atop (s+1) \mid ai}^{s} x^{(k-1)\left(\left\lceil{\frac{ai}{s+1}}\right\rceil-\left\lfloor{\frac{ai}{s+1}}\right\rfloor\right)}\cdot x
+\sum_{i=1\atop (s+1) \nmid ai }^{s} x^{(k-1)\left(\left\lceil{\frac{ai}{s+1}}\right\rceil-\left\lfloor{\frac{ai}{s+1}}\right\rfloor\right)}\cdot x.
\end{align*}
Let $a = b p_1$ and $s+1 = b p_2$, where $\gcd(p_1, p_2) = 1$. Then $s+1 \mid ai$ if and only if $p_2 \mid i$. This implies that the possible values for $i$ are exactly $p_2, 2p_2, \dots, (b-1)p_2$.
we obtain
\begin{align*}
h^*_{\mathcal{R}^{\mathrm{odd}}_{s,k,a}}(x) =1 + (b-1)x + (s+1-b)x^k.
\end{align*}
This completes the proof.
\end{proof}

\begin{cor}
Let $a,s\in\mathbb{Z}_{\geq 0}$.
Let $d = 2k$ for $k \geq 2$, and consider the $d$-dimensional simplex defined as
\begin{align*}
\mathcal{R}^{\mathrm{even}}_{s,k,a}:= \mathrm{conv}\{0, \mathbf e_1, \ldots, \mathbf e_{d-1}, (1,\underbrace{a, \ldots, a}_{k-1} , \underbrace{ -a, \ldots, -a}_{k-1}, s + 1)\} \subset \mathbb{R}^d,
\end{align*}
The $h^*$-polynomial of this simplex is given by
\begin{align*}
h^*_{\mathcal{R}^{\mathrm{even}}_{s,k,a}}(x) = h^*_{\mathcal{R}^{\mathrm{odd}}_{s,k,a}}(x).
\end{align*}
\end{cor}
\begin{proof}
By Theorem~\ref{thm-Delta}, we have
\begin{align*}
h^*_{\mathcal{R}^{\mathrm{even}}_{s,k,a}}(x) &= \sum_{i=0}^{s}x^{\left\lceil{\frac{i}{s+1}}\right\rceil
+\left\lceil{\frac{ai}{s+1}}\right\rceil+\cdots+\left\lceil{\frac{ai}{s+1}}\right\rceil+\left\lceil{\frac{-ai}{s+1}}\right\rceil
+\cdots+\left\lceil{\frac{-a i}{s+1}}\right\rceil+\left\lceil{\frac{(1-1-(k-1)a+(k-1)a) i}{s+1}}\right\rceil}
\\&=1+\sum_{i=1}^{s} {x^{(k-1)\left(\left\lceil{\frac{ai}{s+1}}\right\rceil-\left\lfloor{\frac{ai}{s+1}}\right\rfloor\right)} x^{\left\lceil{\frac{i}{s+1}}\right\rceil}}
\\&=h^*_{\mathcal{R}^{\mathrm{odd}}_{s,k,a}}(x).
\end{align*}
This completes the proof.
\end{proof}

Using an analogous argument, we derive the following result.
\begin{cor}{\em (Extended Reeve simplices)}\label{Extend-Reeve-Simplic}
Let $s\in\mathbb{Z}_{\geq 0}$.
Let $d\geq 3$ be an integer, and consider the $d$-dimensional simplex in $\mathbb{R}^d$ defined as
\begin{align*}
\mathcal{M}_{s,d}:= \mathrm{conv}\{0, \mathbf e_1, \ldots, \mathbf e_{d-1}, \omega\} \subseteq \mathbb{R}^d,
\end{align*}
where
\begin{align*}
\omega=
\begin{cases}
  ( \underbrace{1,\ldots,1}_{k-1}, \underbrace{s, \ldots, s}_{k-1}, s + 1), & \text{if $d=2k-1$}, \\
  ( \underbrace{1,\ldots,1}_{k}, \underbrace{s, \ldots, s}_{k-1}, s + 1), & \text{if $d=2k$}.
\end{cases}
\end{align*}
Then we have
$$h^*_{\mathcal{M}_{s,d}}(x) = sx^k + 1.$$
\end{cor}
\begin{proof}
When $d$ is odd, by Theorem~\ref{thm-Delta}, we have
\begin{align*}
h^*_{\mathcal{M}_{s,d}}(x) &=1+\sum_{i=1}^{s}x^{(k-1)\left(\left\lceil{\frac{i}{s+1}}\right\rceil +\left\lceil{\frac{si}{s+1}}\right\rceil\right)} x^{\left\lceil{\frac{(1-(k-1)-(k-1)s)i}{s+1}}\right\rceil}
\\&=1+\sum_{i=1}^{s}x^{k-1} x^{(k-1)\left\lceil{\frac{(s+1)i-i}{s+1}}\right\rceil} x^{\left\lceil{\frac{(-k+1)(s+1)i+i}{s+1}}\right\rceil}
\\&=1+\sum_{i=1}^{s}x^{k-1} x^{(k-1)i} x^{(-k+1)i+1}
\\&=1+sx^k.
\end{align*}
For the case when $d$ is even, a similar argument applies.
This completes the proof.
\end{proof}

When $d$ is odd, the result of Corollary \ref{Extend-Reeve-Simplic} reduces to \cite[Example 2.5]{Ferroni23}.
Such simplices $\mathcal{M}_{s,d}$ are referred to as \emph{generalized Reeve simplices} and denoted by $\mathcal{R}_{s,k}$.
When $d=3$, these coincide with the Reeve tetrahedra.
The $h^*$-polynomial of $\mathcal{R}_{s,k}$ in \cite[Example 2.5]{Ferroni23} was derived using a result from \cite{Batyrev-Johannes}.

An \emph{empty simplex} is a simplex which contains no lattice point except for its vertices.
A $d$-dimensional simplex having $h^*$-polynomial $h_dx^d+\cdots+h_1x+h_0$ is empty if and only if $h_1=h_d=0$.
Batyrev and Hofscheier \cite{Batyrev-Johannes} proved that for an odd-dimensional empty simplex of dimension $d$, the $h^*$-polynomials look like $1+qx^{\frac{d+1}{2}}$.
Both simplices $\mathcal{R}^{\mathrm{even}}_{s,k,a}$ and $\mathcal{M}_{s,d}$ are empty.

\subsection{Fast algorithm for the $h^*$-polynomials of $\Delta(0,q)$}
Given $q=(q_1,\ldots,q_{d-1},n) \in \mathbb{Z}^d$ with $n\neq 0$, by Theorem \ref{thm-Delta}, the $h^*$-polynomial of $\Delta(0,q)$ is
\begin{align*}
h_{\Delta(0,q)}^*(x)=\sum_{j=0}^{n-1}x^{\left\lceil{\frac{q_1 j}{n}}\right\rceil+\cdots+\left\lceil{\frac{q_d j}{n}}\right\rceil},
\quad \text{ where }\quad q_d=1-\sum_{i=1}^{d-1}q_i.
\end{align*}
This formula is hard to use when $n$ is large.

For further exploration, we develop a fast algorithm for computing the $h^*$-polynomial of $\Delta(0,q)$.
Our algorithm only depend on the numerical values of $q_1, \dots, q_{d-1}$, and is independent of $n$ (no less than $|q_i|$).

Let us introduce several functions.
Let
\begin{align*}
a_i=\frac{n}{q_i}\in \mathbb{Q}\quad \text{and} \quad \varphi_{a_i}(j)=\left\lceil \frac{j}{a_i}\right\rceil \in\mathbb{Z} \quad \text{for} \quad 1\leq i\leq d.
\end{align*}
Note that by assumption, $|a_i|\geq 1$. It is evident that for a fixed $a_i$, $\varphi_{a_i}(j)$ is a piecewise constant function in $j\in \mathbb{R}$.
Its difference function is defined by
\begin{align*}
\nabla\varphi_{a_i}(j)=\varphi_{a_i}(j)-\varphi_{a_i}(j-1) =\left\lceil \frac{j}{a_i}\right\rceil-\left\lceil \frac{j-1}{a_i}\right\rceil  \in\mathbb{Z} \quad \text{for} \quad 1\leq i\leq d.
\end{align*}
Observe that the difference function is almost $0$ everywhere. In particular, we have
\begin{lem}\label{Lemma-c1-1}
For $1\leq i\leq d$ and $m\geq 1$, let
\begin{align*}
c_{i,m} = \begin{cases}
\lfloor (m-1)a_i+1\rfloor & \text{\emph{if}}\quad a_i \geq 1; \\
\lceil -ma_i\rceil & \text{\emph{if}}\quad a_i\leq -1.
\end{cases}
\end{align*}
For $a_i\geq 1$, we have $\nabla\varphi_{a_i}(c_{i,m})=1$ for $m\geq 1$, and $\nabla\varphi_{a_i}(j)=0$ for all positive integers $j\neq c_{i,m}$.
For $a_i\leq -1$, we have $\nabla\varphi_{a_i}(c_{i,m})=-1$ for $m\geq 1$, and $\nabla\varphi_{a_i}(j)=0$ for all positive integers $j\neq c_{i,m}$.
\end{lem}
\begin{proof}
This conclusion follows from the fact that:
when $a_i\geq 1$ and $j$ is a positive real number, the function $\nabla\varphi_{a_i}(j)$ takes value $1$ on the non-overlapping intervals $((m-1)a_i, (m-1)a_i+1]$ and $0$ elsewhere;
when $a_i\leq -1$ and $j$ is a positive real number, the function $\nabla\varphi_{a_i}(j)$ takes value $-1$ on the non-overlapping intervals $[-ma_i, -ma_i+1)$ and $0$ elsewhere. The proof is completed by taking $j$ to be a positive integer.
\end{proof}

By Lemma \ref{Lemma-c1-1}, we have
\begin{align}\label{Formula-Nabla-aIJ}
\sum_{j=1}^{n-1} \nabla\varphi_{a_i}(j) x^j =
\begin{cases}
x^{c_{i,1}}+x^{c_{i,2}}+\cdots +x^{c_{i,q_i-1}} & \text{if}\quad a_i =1; \\
x^{c_{i,1}}+x^{c_{i,2}}+\cdots +x^{c_{i,q_i}} & \text{if}\quad a_i > 1; \\
-x^{c_{i,1}}-x^{c_{i,2}}-\cdots -x^{c_{i,-q_i-1}} & \text{if}\quad a_i \leq -1.
\end{cases}
\end{align}
Let
$$\Phi(j)=\sum_{i=1}^d\varphi_{a_i}(j), \quad j\in\mathbb{N}.$$
Note that $\Phi(0)=0$.
We consider the following two generating functions:
\begin{align*}
F(x)=\sum_{i=1}^d \sum_{j=1}^{n-1}\nabla\varphi_{a_i}(j) x^j =\sum_{j=1}^{n-1} \sum_{i=1}^d \nabla\varphi_{a_i}(j) x^j
\quad \text{and}\quad
G(x)=\sum_{j=0}^{n-1} \Phi(j) x^j=\sum_{j=0}^{n-1} \sum_{i=1}^d\varphi_{a_i}(j) x^j.
\end{align*}

\begin{lem}\label{Lemma-FX-to-GX}
The polynomial $G(x)$ can be reconstructed from the polynomial $F(x)$. In other words, $G(x)$ is determined uniquely by the polynomial $F(x)$.
\end{lem}
\begin{proof}
By $\Phi(0)=0$, we know that the constant term of $G(x)$ is $0$.
We only need to obtain the coefficient of $x^j$ ($1\leq j\leq n-1$) in $G(x)$, i.e., $\Phi(j)$.
Let $[x^j]F(x)$ denote the coefficient of $x^j$ in the polynomial $F(x)$.
According to
\begin{align*}
\Phi(j)-\Phi(j-1)&=[x^j]F(x),
\\ \Phi(j-1)-\Phi(j-2)&=[x^{j-1}]F(x),
\\ \cdots
\\ \Phi(1)-\Phi(0)&=[x^1]F(x),
\end{align*}
we obtain $\Phi(j)=\sum_{\ell=1}^j[x^\ell]F(x)$.
This completes the proof.
\end{proof}

\begin{lem}\label{Lemma-GX-to-HStar}
The polynomial $h_{\Delta(0,q)}^*(x)$ can be reconstructed from the polynomial $G(x)$. In other words, $h_{\Delta(0,q)}^*(x)$ is determined uniquely by the polynomial $G(x)$.
\end{lem}
\begin{proof}
It is clear that
\begin{align*}
h_{\Delta(0,q)}^*(x)&=\sum_{j=0}^{n-1}x^{\sum_{i=1}^d\varphi_{a_i}(j)}
 = \sum_{\ell=0}^d \left|\left\{j: \sum_{i=1}^d\varphi_{a_i}(j)=\ell,\quad 0\leq j\leq n-1 \right\}\right| x^\ell
\\& = \sum_{\ell=0}^d \left|\left\{j: \Phi(j)=\ell,\quad 0\leq j\leq n-1 \right\}\right| x^\ell
\\& = \sum_{\ell=0}^d \left|\left\{j: [x^j]G(x)=\ell,\quad 0\leq j\leq n-1 \right\}\right|  x^\ell.
\end{align*}
This completes the proof.
\end{proof}

\begin{exa}
Let $d=7$, $n=20$, and $q=(1,5,6,8,-3,-7,20)$. Thus we get $q_7=1-(1+5+6+8-3-7)=-9$.
A routine calculation yields
$$F(x)=4x-x^3+x^4-x^7+x^8-x^9+2x^{11}-2x^{12}+2x^{13}-x^{14}-x^{15}+2x^{17}-x^{18}.$$
Furthermore, we obtain
\begin{align*}
G(x)&=0x^0+4x^1+4x^2+3x^3+4x^4+4x^5+4x^6+3x^7+4x^8+3x^9+3x^{10}+5x^{11}
\\ & \quad \quad +3x^{12}+5x^{13}+4x^{14}+3x^{15}+3x^{16}+5x^{17}+4x^{18}+4x^{19}.
\end{align*}
Therefore, the $h^*$-polynomial of $\Delta(0,q)$ is
\begin{align*}
h_{\Delta(0,q)}^*(x)=1+7x^3+9x^4+3x^5.
\end{align*}

However, if we  explicitly  construct $G(x)$, we still have $n$ terms to deal with. Our point is that
we can construct the $h^*$-polynomial directly from $F(x)$, which has at most $|q_1|+\cdots + |q_d|$ terms.

Indeed, if $F(x) = \sum_{j=1}^{s-1} c_{j} x^{i_j}$, then
$h_{\Delta(0,q)}^*(x)$ can be computed by the following facts: $[x^0] G(x)=0$, and the coefficient remains $0$ until $[x^{i_1}] G(x)=c_1$;
then the coefficients remains $c_1$ until $[x^{i_2}] G(x)=c_1+c_2$, and so on. In general, $[x^{i_j}] G(x)=c_1+c_2+\cdots +c_j$. Then
\begin{align}\label{Equ-LXT-H*}
h_{\Delta(0,q)}^*(x)=(i_1-i_0)x^0 + (i_2-i_1)x^{c_1}+ \cdots + (i_s-i_{s-1}) x^{c_1+\cdots+c_{s-1}},
\end{align}
where we treat $i_0=0, \ i_s=n$.
\end{exa}

\begin{algorithm}\label{TheAlgorithm-HStar}
\DontPrintSemicolon
\KwInput{$q=(q_1,\ldots,q_{d-1},n) \in \mathbb{Z}^d$ with $q_d=1-(q_1+q_2+\cdots+q_{d-1})$ and $n\geq |q_i|$.}
\KwOutput{The $h^*$-polynomial of $\Delta(0,q)$.}

Computing $a_i=\frac{n}{q_i}$ for $1\leq i\leq d$.

Computing the polynomial $F(x)$ by Equation \eqref{Formula-Nabla-aIJ}.

Return the $h^*$-polynomial of $\Delta(0,q)$ by Equation \ref{Equ-LXT-H*}.

\caption{Computing the $h^*$-polynomial of $\Delta(0,q)$.}
\end{algorithm}

\begin{thm}
Let $q=(q_1,\ldots,q_{d-1},n) \in \mathbb{Z}^d$ with $q_d=1-(q_1+q_2+\cdots+q_{d-1})$ and $n\geq |q_i|$.
Then Algorithm \ref{TheAlgorithm-HStar} correctly computes the $h^*$-polynomial of $\Delta(0,q)$
with $ \mathcal O(\sum_{i=1}^{d-1}|q_i|)$ arithmetic operations.
\end{thm}

\begin{proof}
It follows from Equation \eqref{Formula-Nabla-aIJ}, together with Lemmas \ref{Lemma-FX-to-GX} and \ref{Lemma-GX-to-HStar}, that Algorithm \ref{TheAlgorithm-HStar} correctly computes the $h^*$-polynomial of $\Delta(0,q)$.

Furthermore, it is straightforward to see that the overall time complexity of Algorithm \ref{TheAlgorithm-HStar} is dominated by Step 3, which, by Equation \ref{Equ-LXT-H*}, is bounded by $\mathcal{O}\left(\sum_{i=1}^{d} |q_i|\right)$ arithmetic operations. The theorem then follows from the expression of $q_d$.
\end{proof}

Algorithm \ref{TheAlgorithm-HStar} has been implemented by us in \texttt{Maple} \cite{Maple} and is available at:
\url{https://github.com/TygerLiu/TygerLiu.github.io/tree/main/Procedure/Sign-pattern-problem}.

A modified version of Algorithm \ref{TheAlgorithm-HStar} is used in later sections to compute the $h^*$-polynomial of a particular family of simplices. This plays a key role in the investigation of the sign patterns of the Ehrhart polynomial.

\section{Simplex family: $\Delta(0,q^{(m)})$}\label{Section-Five}

We  now introduce a family of parametric $\Delta(0,q)$-simplices, namely $\Delta(0,q^{(m)})$. This family has nicer $h^*$-polynomials.

\subsection{Characteristic polynomials}
Consider $q=(q_1, \dots, q_{d-1},n)\in \mathbb{Z}^d$, where $n$ is a positive integer. We further require
$q_i \mid n$ for all $i = 1, \dots, d$, with the convention $q_{d}=1-\sum_{j=1}^{d-1}q_j$.
For a positive integer $m$, define $q^{(m)} = (q_1, \dots, q_{d-1}, mn)$. We study the family of $\Delta(0,q^{(m)})$-simplices.
We will often use the notation $a^+ = \max(0,a)$ for any integer $a$.

\begin{thm}\label{thm-dec-UnEX}
There exist unique polynomials $L_1(x)$ and $L_2(x)$ such that
\begin{align*}
h^*_{\Delta(0,q^{(m)})}(x)=mx\cdot L_1(x) +L_2(x),
\end{align*}
where
\begin{align*}
L_1(x)=\sum_{j=0}^{n-1}x^{\left\lceil{\frac{q_1j+q_1^+}{n}}\right\rceil+\cdots+\left\lceil{\frac{q_d j+q_d^+}{n}}\right\rceil-1}
\end{align*}
and
\begin{align*}
L_2(x)=h^*_{\Delta(0,q)}(x)-\sum_{j=0}^{n-1}x^{\left\lceil{\frac{q_1j+q_1^+}{n}}\right\rceil
  +\cdots+\left\lceil{\frac{q_d j+q_d^+}{n}}\right\rceil}.
\end{align*}
\end{thm}
\begin{proof}
Since $q_k \mid n$, let $a_k = n/q_k$ (which is an integer) with $1\leq k\leq d$. By Theorem~\ref{thm-Delta}, we have
\begin{align*}
h^*_{\Delta(0,q^{(m)})}(x)&= \sum_{i=0}^{mn-1}x^{\left\lceil{\frac{q_1 i}{mn}}\right\rceil+\cdots+\left\lceil{\frac{q_d i}{mn}}\right\rceil}\\
&= 1+\sum_{j=0}^{n-1}\sum_{i=mj+1}^{m(j+1)-1}x^{\left\lceil{\frac{q_1 i}{mn}}\right\rceil+\cdots+\left\lceil{\frac{q_d i}{mn}}\right\rceil} +\sum_{j=1}^{n-1}x^{\left\lceil{\frac{q_1 mj}{mn}}\right\rceil+\cdots+\left\lceil{\frac{q_d mj}{mn}}\right\rceil}\\
&= 1+\sum_{j=0}^{n-1}\sum_{i=mj+1}^{mj+m-1}x^{\left\lceil{\frac{ i}{ma_1}}\right\rceil+\cdots+\left\lceil{\frac{ i}{ma_d}}\right\rceil} +\sum_{j=1}^{n-1}x^{\left\lceil{\frac{q_1 j}{n}}\right\rceil+\cdots+\left\lceil{\frac{q_d j}{n}}\right\rceil}.
\end{align*}
It is clear that for every $1\leq k\leq d$ and $0\leq j\leq n-1$, we get
$$x^{\left\lceil{\frac{mj+1}{ma_k}}\right\rceil}=x^{\left\lceil{\frac{mj+2}{ma_k}}\right\rceil}
=\cdots=x^{\left\lceil{\frac{mj+m-1}{ma_k}}\right\rceil}.$$
Moreover, when $a_k>0$, (now $q_k^+=q_k$), we have
$$x^{\left\lceil{\frac{mj+m-1}{ma_k}}\right\rceil}=x^{\left\lceil{\frac{mj+m}{ma_k}}\right\rceil}
=x^{\left\lceil{\frac{q_kmj+mq_k^+}{mn}}\right\rceil},$$
while for $a_k<0$, (now $q_k^+=0$), we have
$$x^{\left\lceil{\frac{mj+1}{ma_k}}\right\rceil}=x^{\left\lceil{\frac{mj}{ma_k}}\right\rceil}
=x^{\left\lceil{\frac{q_kmj+mq_k^+}{mn}}\right\rceil}.$$
Therefore, we obtain
\begin{align*}
h^*_{\Delta(0,q^{(m)})}(x)&= \sum_{j=0}^{n-1}(m-1)\left(x^{\left\lceil{\frac{q_1 mj+mq_1^+}{mn}}\right\rceil+\cdots+\left\lceil{\frac{q_d mj+mq_d^+}{mn}}\right\rceil}\right) + h^*_{\Delta(0,q)}(x)
\\=&m\sum_{j=0}^{n-1}x^{\left\lceil{\frac{q_1j+q_1^+}{n}}\right\rceil+\cdots+\left\lceil{\frac{q_d j+q_d^+}{n}}\right\rceil}+
\left( h^*_{\Delta(0,q)}(x)-\sum_{j=0}^{n-1}x^{\left\lceil{\frac{q_1j+q_1^+}{n}}\right\rceil+\cdots+\left\lceil{\frac{q_d j+q_d^+}{n}}\right\rceil}\right)
\\=&mx\cdot L_1(x) +L_2(x).
\end{align*}

Since $m$ is arbitrary and $h^*_{\Delta(0,q^{(m)})}(x)$ is a polynomial with constant term $1$, it follows that $L_1(x)$ and $L_2(x)$ are uniquely determined polynomials.
\end{proof}

We call $L_1(x)$ and $L_2(x)$ the \emph{characteristic polynomials} of the $\Delta(0,q^{(m)})$ simplex family.
An efficient algorithm similar to Algorithm \ref{TheAlgorithm-HStar} also exists for computing the polynomial $L_1(x)$. It can be obtained through a minor adaptation of Algorithm \ref{TheAlgorithm-HStar}, the details of which we omit here. We have implemented this algorithm in \texttt{Maple}, which is available at the same address:
\url{https://github.com/TygerLiu/TygerLiu.github.io/tree/main/Procedure/Sign-pattern-problem}.

Let $L(x)=xL_1(x)$. We now provide an alternative proof that $L(x)$ is a polynomial of degree at most $d$ and at least $1$.
\begin{prop}
Let $n$ be a positive integer.
For any given vector $q=(q_1, \dots, q_{d-1},n)\in \mathbb{Z}^d$, assume that $q_i \mid n$ for every $i = 1, \dots, d$,
where we set
$$q_{d}=1-\sum_{j=1}^{d-1}q_j\quad \text{and}\quad q_{d+1}=-1.$$
Then
$$L(x)=\sum_{j=0}^{n-1}x^{\left\lceil{\frac{q_1j+q_1^+}{n}}\right\rceil+\cdots+\left\lceil{\frac{q_d j+q_d^+}{n}}\right\rceil}$$
is a polynomial of degree at most $d$ and at least $1$.
\end{prop}
\begin{proof}
Since $q_{d+1}=-1$, we have $q_1+\cdots +q_{d}+q_{d+1}=0$ and
$$\left\lceil \frac{q_{d+1}j+q_{d+1}^+}{n}\right\rceil=0 \quad \text{for all } 0\leq j\leq n-1.$$
Let $\{u_1,u_1,\ldots,u_k\}$ and $\{-v_{k+1},\ldots,-v_{d+1}\}$ be the sets of positive and negative integers in $\{q_1,\ldots,q_d,q_{d+1}\}$, respectively.
Then we have
\begin{align*}
\sum_{i=1}^d \left\lceil{\frac{q_ij+q_i^+}{n}}\right\rceil
=\sum_{i=1}^k \left\lceil{\frac{u_i(j+1)}{n}}\right\rceil+ \sum_{i=k+1}^{d+1}\left\lceil{\frac{-v_i j}{n}}\right\rceil\quad \text{for all }\quad 0\leq j\leq n-1.
\end{align*}
For any $a>0$ and $j\geq 0$, it is known that the floor function satisfies
$\left\lceil \frac{j+1}{a}\right\rceil =1+\left\lfloor \frac{j}{a}\right\rfloor$.
By $q_i\mid n$ for all $i$, we obtain
\begin{align*}
\sum_{i=1}^k \left\lceil{\frac{u_i(j+1)}{n}}\right\rceil+ \sum_{i=k+1}^{d+1}\left\lceil{\frac{-v_i j}{n}}\right\rceil
=k+\sum_{i=1}^k \left\lfloor{\frac{u_i j}{n}}\right\rfloor- \sum_{i=k+1}^{d+1}\left\lfloor{\frac{v_i j}{n}}\right\rfloor
\quad \text{for all }\quad 0\leq j\leq n-1.
\end{align*}
For any $a\in \mathbb{R}$, let $\{ a\}=a-\lfloor a\rfloor$.
Since $u_1+\cdots+u_k-v_{k+1}-\cdots -v_{d+1}=0$, it follows that
\begin{align*}
\sum_{i=1}^k \left\lceil{\frac{u_i(j+1)}{n}}\right\rceil+ \sum_{i=k+1}^{d+1}\left\lceil{\frac{-v_i j}{n}}\right\rceil
&=k+\sum_{i=1}^k \left(\frac{u_i j}{n}- \left\{\frac{u_i j}{n}\right\}\right)
- \sum_{i=k+1}^{d+1}\left(\frac{v_i j}{n}- \left\{\frac{v_i j}{n}\right\}\right)
\\&=\sum_{i=1}^k \left(1- \left\{\frac{u_i j}{n}\right\}\right)+\sum_{i=k+1}^{d+1}\left\{\frac{v_i j}{n}\right\}.
\end{align*}
According to
$$0< 1- \left\{\frac{u_i j}{n}\right\} \leq 1, \quad 0\leq \left\{\frac{v_i j}{n}\right\} <1, \quad \text{and}\quad \sum_{i=1}^k u_i-\sum_{i=k+1}^{d+1}v_i=0, $$
for $0\leq j\leq n-1$, we obtain
\begin{align*}
1\leq \sum_{i=1}^k \left(1- \left\{\frac{u_i j}{n}\right\}\right)+\sum_{i=k+1}^{d+1}\left\{\frac{v_i j}{n}\right\}<k+(d+1-k-1+1)=d+1.
\end{align*}
This completes the proof.
\end{proof}

\subsection{Associated properties}
Follow the notation in the previous subsection.

\begin{prop}\label{Prop-L1L2-pro}
The polynomials $L_1(x)$ and $L_2(x)$ satisfy
\begin{align*}
L_1(1)=n, \quad L_2(0)=1,\quad\text{and}\quad L_2(1)=0.
\end{align*}
\end{prop}
\begin{proof}
By Equation~\eqref{formula-h^*-prop}, we have $h^*_{\Delta(0,q^{(m)})}(0)=1$, this imples $L_2(0)=1$.
By Equation \eqref{formual--cone(P)hstar} and \cite[Lemma 9.2]{BeckRobins}, we know that $$h^*_{\Delta(0,q^{(m)})}(1)=\#(\Pi(\mathrm{cone}(\Delta(0,q^{(m)})))\cap \mathbb{Z}^{d+1})=mn.$$
Since $h^*_{\Delta(0,q^{(m)})}(1)=m\cdot L_1(1) +L_2(1)=mn$ for any $m$, we have $L_1(1) = n$ and $L_2(1) = 0$.
\end{proof}

\begin{prop}\label{Proposition-Symmetric-L1}
The characteristic polynomial $L_1(x)=\sum_{i=0}^{d-1} c_i x^i$ is a symmetric polynomial, i.e., $c_i=c_{d-1-i}$ for all $i=0,\ldots,d-1$.
\end{prop}
\begin{proof}
Let $L(x)=xL_1(x)$. To prove that the polynomial $L_1(x)$ is symmetric, we need to show that $L(x) = x^{d+1} L(1/x)$.
Recall that
\begin{align*}
L(x) = \sum_{j=0}^{n-1} x^{A(j)},\quad \text{where}\quad
A(j) = \left\lceil{\frac{q_1j+q_1^+}{n}}\right\rceil+\cdots+\left\lceil{\frac{q_d j+q_d^+}{n}}\right\rceil.
\end{align*}
Since $q_i \mid n$ for all $i$ and $q_d = 1 - (q_1 + \cdots + q_{d-1})$, we have $\sum_{i=1}^{d} q_i = 1$.
Let $\{q_{k_1},\ldots,q_{k_s}\}$ and $\{q_{k_{s+1}},\ldots,q_{k_d}\}$ be the sets of positive and negative integers in $\{q_1,\ldots,q_d\}$, respectively.
Then we have
\begin{align*}
A(n-1-j) &= \left\lceil{\frac{q_1(n-1-j)+q_1^+}{n}}\right\rceil+\cdots+\left\lceil{\frac{q_d (n-1-j)+q_d^+}{n}}\right\rceil
\\&= \sum_{i=1}^{d} q_i + \sum_{i=1}^{s} \left\lceil -\frac{q_{k_i} j}{n} \right\rceil + \sum_{i=s+1}^{d} \left\lceil -\frac{q_{k_i} (j+1)}{n} \right\rceil.
\end{align*}
Using the identity $\lceil -r \rceil = -\lfloor r \rfloor$ and the fact that $\sum_{i=1}^{d} q_i = 1$, we obtain
$$A(n-1-j) = 1 - \left( \sum_{i=1}^{s} \left\lfloor \frac{q_{k_i} j}{n} \right\rfloor + \sum_{i=s+1}^{d} \left\lfloor \frac{q_{k_i} (j+1)}{n} \right\rfloor \right).$$
Let
$$B(j) = \sum_{i=1}^{s} \left\lfloor \frac{q_{k_i} j}{n} \right\rfloor + \sum_{i=s+1}^{d} \left\lfloor \frac{q_{k_i} (j+1)}{n} \right\rfloor.$$
Since $q_i \mid n$, we set $a_i = n/q_i$ (which is an integer).
Now consider the difference $A(j) - B(j)$:
$$A(j) - B(j) = \left( \sum_{i=1}^{s} \left\lceil \frac{q_{k_i} (j+1)}{n} \right\rceil + \sum_{i=s+1}^{d} \left\lceil \frac{q_{k_i} j}{n} \right\rceil \right) - \left( \sum_{i=1}^{s} \left\lfloor \frac{q_{k_i} j}{n} \right\rfloor + \sum_{i=s+1}^{d} \left\lfloor \frac{q_{k_i} (j+1)}{n} \right\rfloor \right)=d,$$
since for $i = 1, \dots, s$, ($a_{k_i}>0$), we have
$$\left\lceil \frac{q_{k_i} (j+1)}{n} \right\rceil - \left\lfloor \frac{q_{k_i} j}{n} \right\rfloor = \left\lceil \frac{j+1}{a_{k_i}} \right\rceil - \left\lfloor \frac{j}{a_{k_i}} \right\rfloor = 1,$$
and for $i = s+1, \dots, d$,  ($a_{k_i}<0$), we have
$$\left\lceil \frac{q_{k_i} j}{n} \right\rceil - \left\lfloor \frac{q_{k_i} (j+1)}{n} \right\rfloor = \left\lceil \frac{j}{a_{k_i}} \right\rceil - \left\lfloor \frac{j+1}{a_{k_i}} \right\rfloor = 1.$$
It follows that
$$A(j) + A(n-1-j) = A(j) + (1 - B(j)) = 1 + (A(j) - B(j)) = 1 + d.$$
We now obtain
\begin{align*}
L(x)= \sum_{j=0}^{n-1} x^{A(j)} = \sum_{j=0}^{n-1} x^{A(n-1-j)} = \sum_{j=0}^{n-1} x^{d+1 - A(j)}
= x^{d+1} \sum_{j=0}^{n-1} x^{-A(j)} = x^{d+1} L(1/x).
\end{align*}
Therefore, $L(x)$ (i.e., $L_1(x)$) is a symmetric polynomial. The conclusion holds.
\end{proof}

Note that the polynomial $L_1(x)$ need not be unimodal. For example, when
$$(q_1,q_2,q_3,q_4,q_5,q_6,q_7)=(1,2,3,3,4,-5,-7)$$
and $n=420$, we have $L_1(x)=159x^4+102x^3+159x^2$.

\begin{prop}
The characteristic polynomial $L_1(x)$ is unimodal for $d=3,4$.
\end{prop}
\begin{proof}
Let $h^*_{\Delta(0,q^{(m)})}(x)=mx\cdot L_1(x) +L_2(x)=1+h_1^*x+\cdots+h_d^*x^d$.
By Hibi’s lower bound theorem \cite{Hibi1994}, we have $h_1^*\leq h_i^* \text{ for all } 1\leq i\leq d-1$.
Let $mxL_1(x)=m(c_1x+c_2x^2+\cdots+c_dx^d)$.
For $d=3$, if $L_1(x)$ is not unimodal, then it must admit $c_1>c_2<c_3$. For sufficiently large $m$, this yields $h_1^*>h_2^*$, leading to a contradiction.
For $d=4$, if $L_1(x)$ is not unimodal, then it must admit $c_1>c_2=c_3<c_4$ (by Proposition \ref{Proposition-Symmetric-L1}). For sufficiently large $m$, this yields $h_1^*>h_2^*$, leading to a contradiction.
Then the conclusion follows.
\end{proof}

We now present two interesting families $\Delta(0,q^{(m)})$ with closed formula.
\begin{cor}
Let $q=(-1,-1, \dots, -1,d)\in \mathbb{Z}^d$.
Then we have
$$h^*_{\Delta(0,q^{(m)})}(x)=m\sum_{i=1}^d x^i+(1-x^d).$$
\end{cor}
\begin{proof}
Since $n=d$ and $q_{d}=d$, we have
\begin{align*}
  xL_1(x) & =\sum_{j=0}^{d-1} x^{\lceil \frac{-j}{d}\rceil(d-1) + \lceil j+1\rceil} = \sum_{i=1}^d x^i \\
  L_2(x) & =\sum_{j=0}^{d-1} x^{\lceil \frac{-j}{d}\rceil(d-1) + \lceil j\rceil} - \sum_{i=1}^d x^i = 1-x^d.
\end{align*}
The conclusion follows.
\end{proof}

A simplex whose $h^*$-polynomial is $\sum_{i=0}^d x^i$ is constructed in \cite[Example 2.7]{Ferroni23}.

\begin{cor}
Let $q=(-2^0,-2^1, \dots, -2^{d-2},2^{d-1})\in \mathbb{Z}^d$.
Then we have
$$h^*_{\Delta(0,q^{(m)})}(x)=((m-1)x+1)(1+x)^{d-1}.$$
\end{cor}
\begin{proof}
Since $n=2^{d-1}$ and $q_{d}=2^{d-1}$, we have
\begin{align*}
xL_1(x)&=\sum_{j=0}^{2^d-1} x^{\left\lceil\frac{-j}{2^{d-1}}\right\rceil+\left\lceil\frac{-2j}{2^{d-1}}\right\rceil+\cdots + \left\lceil\frac{-2^{d-2}j}{2^{d-1}}\right\rceil +\left\lceil\frac{2^{d-1}j+2^{d-1}}{2^{d-1}}\right\rceil}
\\& =\sum_{j=0}^{2^d-1} x^{-\left(\left\lfloor \frac{j}{2^{d-1}}\right\rfloor+\left\lfloor \frac{j}{2^{d-2}}\right\rfloor +\cdots + \left\lfloor \frac{j}{2}\right\rfloor\right)+(j+1)}.
\end{align*}
Let $j=\sum_{k=0}^{d-2} b_k 2^k$ be the binary representation of $j$ ($0\leq j\leq 2^{d-1}-1$), where $b_k\in\{0,1\}$.
Therefore, we get
$$\sum_{m=1}^{d-1}\left\lfloor \frac{j}{2^m}\right\rfloor = \sum_{m=1}^{d-1}\sum_{k=m}^{d-2} b_k 2^{k-m}
=\sum_{k=1}^{d-2} b_k \sum_{m=1}^k 2^{k-m}=\sum_{k=1}^{d-2} b_k(2^k-1).$$
It follows that
$$xL_1(x)=\sum_{j=0}^{2^d-1} x^{-\sum_{k=1}^{d-2} b_k(2^k-1)+\sum_{k=0}^{d-2} b_k 2^k + 1}
=\sum_{j=0}^{2^d-1} x^{1+\sum_{k=0}^{d-2} b_k}.$$
Because $\sum_{k=0}^{d-2} b_k$ is the number of 1's in the binary representation of $j$, we have
$$xL_1(x)=x\sum_{j=0}^{2^d-1} x^{\sum_{k=0}^{d-2} b_k}=x\sum_{i=0}^{d-1}\binom{d-1}{i} x^i=x(1+x)^{d-1}.$$
On the other hand, the polynomial $L_2(x)$ is given by
\begin{align*}
L_2(x)&=\sum_{j=0}^{2^d-1} x^{-(\lfloor \frac{j}{2^{d-1}}\rfloor+\lfloor \frac{j}{2^{d-2}}\rfloor +\cdots + \lfloor \frac{j}{2}\rfloor)+j}-x(1+x)^{d-1}
\\ &= (1+x)^{d-1}-x(1+x)^{d-1}
\\ &= (1-x)(1+x)^{d-1}.
\end{align*}
Thus, the $h^*$-polynomial of $\Delta(0,q^{(m)})$ is given by
$$h^*_{\Delta(0,q^{(m)})}(x)=mxL_1(x)+L_2(x)=((m-1)x+1)(1+x)^{d-1}.$$
This completes the proof.
\end{proof}

\section{Eulerian simplex family: $\mathcal{R}_d(m)$}\label{Section-Six}

\subsection{A new formula for Eulerian polynomials}

Let $\mathfrak{S}_d$ be the set of all permutations on $d$ elements. For a permutation $\pi \in \mathfrak{S}_d$, a \emph{descent} is an index $j \in \{1, \dots, d-1\}$ such that $\pi_j > \pi_{j+1}$, and we write $\mathrm{des}(\pi)$ for the number of descents of $\pi$. The \emph{Eulerian number} $A(d,i)$ counts permutations in $\mathfrak{S}_d$ with exactly $i-1$ descents. The \emph{Eulerian polynomial} is then
$$A_d(x) = \sum_{i=1}^d A(d,i) x^i = \sum_{\pi \in \mathfrak{S}_d} x^{1+\mathrm{des}(\pi)}.$$

It is well known \cite[P292, Exercise 3]{Comtet} that all roots of $A_d(x)/x$ are simple, real, and strictly negative.

\begin{prop}{\em \cite[Proposition 1.1.4]{RP.Stanley}}\label{Generating Function for Eulerian Polynomials}
For every $d\geq 0$, the Eulerian polynomials satisfy the following generating function identity
$$\sum_{t \geq 0} t^d x^t = \frac{A_d(x)}{(1-x)^{d+1}}.$$
\end{prop}

Let $[0,d]_\Z=\{0,1,\ldots,d\}$. There exists a one-to-one correspondence between the $\mathfrak{S}_d$ and $[0,d!-1]_\Z$.
A pair $(\pi_i,\pi_j)$ is called an \emph{inversion} of $\pi$ if $i<j$ and $\pi_i>\pi_j$.
Let $c_i=\#\{i<j \mid \pi_i>\pi_j \}$ for $1\leq i\leq d$.
The sequence $I(\pi)=(c_1,c_2,\ldots,c_d)$ is called the \emph{inversion table} (or \emph{Lehmer code}) of $\pi$.
Therefore, for a given $\pi$, the number of descents of $\pi$ equals the number of pairs with $c_i>c_{i+1}$.

\begin{lem}{\em \cite[Proposition 1.3.12]{RP.Stanley}}\label{Permu-to-inverTab}
Let
$$\mathcal{I}_d=\{(c_1,\ldots, c_d)\mid 0\leq c_i\leq d-i \}=[0,d-1]_\Z\times [0,d-2]_\Z\times \cdots \times [0,0]_\Z.$$
The map $I: \mathfrak{S}_d\rightarrow \mathcal{I}_d$ that sends each permutation to its inversion table is a bijection.
\end{lem}

For a given $(c_1,c_2,\ldots, c_d)$, we briefly describe how to return $\pi=I^{-1}((c_1,c_2,\ldots, c_d))\in \mathfrak{S}_d$.
To construct this permutation, begin with $\pi=()$ and the initial available set $\{1,2,\ldots,d\}$.
We construct $\pi$ by successively choosing the $(c_i+1)$-th smallest remaining element as $\pi_i$ and removing it from the set.
For example, given $(c_1,c_2,c_3,c_4,c_5)=(2,0,0,1,0)$, we return $\pi=31254$.

\begin{lem}
Given a permutation $\pi\in\mathfrak{S}_d$, suppose the inversion table of $\pi$ is $I(\pi)=(c_1,c_2,\ldots,c_d)$.
Let
$$\aleph(\pi)=c_1\cdot (d-1)!+c_2\cdot (d-2)!+\cdots +c_{d-1}\cdot 1! + c_d\cdot 0!.$$
The map $\aleph: \mathfrak{S}_d\rightarrow [0,d!-1]_\Z$ that sends each permutation to $[0,d!-1]_\Z$ is a bijection.
\end{lem}
\begin{proof}
It is clear that $\aleph(\pi)\in [0,d!-1]_\Z$.
For any integer $j\in [0,d!-1]_\Z$, by repeatedly applying the Euclidean division algorithm with respect to $(d-1)!, (d-2)!, \dots, 1!$, we obtain the factorial base coefficients: $c_1, c_2, \dots, c_{d-1}$.
The proof is now complete by Lemma \ref{Permu-to-inverTab}.
\end{proof}

\begin{thm}
Let $0\leq N < d!$ with $d\geq 2$. Then we have
\begin{align}\label{Descent-Aleph-formul}
\mathrm{des}(\aleph^{-1}(N))=N-\sum_{i=0}^{d-2} \left\lfloor \frac{N}{i!(i+2)} \right\rfloor.
\end{align}
\end{thm}
\begin{proof}
The proof is by induction on $d$.
When $d=2,3$, the equality is easily verified.
Assuming the formula holds for $d-1$, we prove it for $d$.
For convenience, we denote the left-hand side of Equation \eqref{Descent-Aleph-formul} by $T_d(N)$.

Let $m=\left\lfloor \frac{N}{(d-1)!}\right\rfloor$ and $r=N-m\cdot (d-1)!$. We know that $0\leq m< d$ and $0\leq r <(d-1)!$.
Assume $\aleph^{-1}(N)=\pi$ and $I(\pi)=(c_1,c_2,\ldots,c_d)$, given that $c_1=m$ and $c_2=\left\lfloor \frac{r}{(d-2)!}\right\rfloor$.
Since the number of descents of $\pi$ equals the number of pairs with $c_i>c_{i+1}$,
we obtain
$$T_d(N)=\chi(c_1>c_2)+T_{d-1}(r).$$

The inequality $c_1>c_2$ is equivalent to $r<m\cdot (d-2)!$, which means
$$N<m\cdot (d-1)!+m\cdot (d-2)!=md\cdot (d-2)!.$$
Therefore, we obtain
\begin{align*}
T_d(N)&=\chi(N<md\cdot (d-2)!)+T_{d-1}(r)
\\ &=\chi(N<md\cdot (d-2)!)+ r- \sum_{i=0}^{d-3} \left\lfloor \frac{r}{i!(i+2)}\right\rfloor.
\end{align*}
To prove
$$T_d(N)=m\cdot (d-1)!+r- \sum_{i=0}^{d-2}\left\lfloor \frac{N}{i!(i+2)}\right\rfloor,$$
it suffices to show that
$$m\cdot (d-1)!- \sum_{i=0}^{d-2}\left\lfloor \frac{m\cdot (d-1)!+r}{i!(i+2)}\right\rfloor+\sum_{i=0}^{d-3} \left\lfloor \frac{r}{i!(i+2)}\right\rfloor=\chi(N<md\cdot (d-2)!).$$
We denote the left-hand side of the above equation by $\mathrm{LHS}$.
Since for $0\leq i\leq d-3$, $\frac{m\cdot (d-1)!}{i!(i+2)}$ is an integer, we have
\begin{align*}
\mathrm{LHS}&=m\cdot (d-1)!- \left\lfloor \frac{m\cdot (d-1)! +r}{(d-2)!d}\right\rfloor - \sum_{i=0}^{d-3} \frac{m\cdot (d-1)!}{i!(i+2)}
\\ &= m\cdot (d-1)!- \left\lfloor \frac{m\cdot (d-1)! +r}{(d-2)!d}\right\rfloor -m\cdot (d-1)! \left( \sum_{i=0}^{d-3} \left(\frac{1}{(i+1)!}-\frac{1}{(i+2)!}\right)\right)
\\ &= m\cdot (d-1)!- \left\lfloor \frac{N}{(d-2)!d}\right\rfloor -m\cdot (d-1)!\left(1-\frac{1}{(d-1)!}\right)
\\ &= m - \left\lfloor \frac{N}{(d-2)!d}\right\rfloor.
\end{align*}

By $N=m\cdot (d-1)!+r$, we set $\frac{N}{(d-1)!}=m+\delta$. Thus $0\leq \delta=\frac{r}{(d-1)!}<1$.
We consider the following two cases:

1.) When $m=0$, we have $\mathrm{LHS}=0=\chi(N<md\cdot (d-2)!)$.

2.) When $m\geq 1$, we consider the value of
$$\left\lfloor \frac{N}{(d-2)!d}\right\rfloor=\left\lfloor \frac{N}{(d-1)!}\cdot \frac{d-1}{d}\right\rfloor
=\left\lfloor \frac{m(d-1)}{d} +\frac{r}{(d-1)!}\cdot\frac{d-1}{d} \right\rfloor
=\left\lfloor \frac{m(d-1)}{d} +\delta \cdot\frac{d-1}{d} \right\rfloor.$$
It is clear that the integer part of $\frac{m(d-1)}{d}$ is $m-1$, and its fractional part is $1-\frac{m}{d}$.
Therefore, we have
$$\left\lfloor \frac{N}{(d-2)!d}\right\rfloor=m-1+\left\lfloor 1-\frac{m}{d} +\delta \cdot\frac{d-1}{d} \right\rfloor.$$
Obviously, $\left\lfloor 1-\frac{m}{d} +\delta \cdot\frac{d-1}{d} \right\rfloor\in \{0,1\}$.
Observe that $1-\frac{m}{d} +\delta \cdot\frac{d-1}{d}\geq 1$ is equivalent to $\delta\geq \frac{m}{d-1}$. It follows that
\begin{align*}
\left\lfloor \frac{N}{(d-2)!d}\right\rfloor=
\begin{cases}
m-1 & \text{ if } \delta<\frac{m}{d-1}, \\
m & \text{ if } \delta\geq \frac{m}{d-1}.
\end{cases}
\end{align*}
Since $\delta<\frac{m}{d-1}$ is equivalent to $r<m\cdot (d-2)!$ (i.e., $N<md\cdot (d-2)!$), we obtain
\begin{align*}
\mathrm{LHS}&=
\begin{cases}
1 & \text{ if } N<md\cdot (d-2)!, \\
0 & \text{ if } N\geq md\cdot (d-2)!
\end{cases}
\\& =\chi(N<md\cdot (d-2)!).
\end{align*}
We have thus finished the proof.
\end{proof}

A direct consequence of the above result is an explicit formula for the Eulerian polynomials.
\begin{cor}\label{Cor-Adx=descent}
The Eulerian polynomial $A_d(x)$ is given by
$$A_d(x)=\sum_{\pi\in \mathfrak{S}_d} x^{1+\mathrm{des}(\pi)}= \sum_{j=0}^{d!-1} x^{j+1-\sum_{i=0}^{d-2}\left\lfloor \frac{j}{i!(i+2)}\right\rfloor}.$$
\end{cor}

\subsection{Eulerian simplex family}

Based on our previously implemented algorithm for computing polynomial $L_1(x)$, we can construct some Eulerian simplex families for small $d$ through a bounded search. This also inspires us to construct a family of Eulerian simplex family in any dimension.

\begin{thm}\label{Prop-Sk}
For every positive integer $m$ and $d\geq 1$, the following family $\mathcal{S}_d(m)$ of polytopes fall into the Eulerian simplex family $\mathcal{R}_d(m)$ of dimension $d$:
\begin{align*}
\mathcal{S}_d(m)=
\begin{cases}
\ell_{m}=[0,m]=\{\alpha\in \mathbb{R} \mid 0\leq \alpha\leq m\} & \text{ if }\quad d=1, \\
\mathrm{conv}\{\mathbf{0}^d, \mathbf{e}_1^d, \mathbf{e}_2^d, \ldots, \mathbf{e}_{d-1}^d, (q_1(d), q_2(d), \ldots, q_{d-1}(d),d!\cdot m)\} & \text{ if }\quad d\geq 2,
\end{cases}
\end{align*}
where $\mathbf{e}_i^d$ denote the $i$-th unit coordinate vector of $\mathbb{R}^d$; $\mathbf{0}^d$ denote the origin of $\mathbb{R}^d$; and
$$q_i(d)=\frac{-d!}{i!+(i-1)!}\quad \text{for} \quad 1\leq i\leq d-1.$$
\end{thm}
\begin{proof}
Let $q_d(d)=1-\sum_{j=1}^{d-1}q_j(d)$. Then we have
\begin{align*}
q_d(d)&=1-d!\cdot \sum_{j=1}^{d-1}\frac{-1}{j!+(j-1)!}=1+d!\cdot \sum_{j=1}^{d-1} \frac{1}{(j-1)!(j+1)}
\\ &= 1+d! \cdot \sum_{j=0}^{d-2} \frac{j+1}{(j+2)!}
= 1+d! \cdot \sum_{i=2}^d \left( \frac{1}{(i-1)!}-\frac{1}{i!} \right)
\\ &= d!.
\end{align*}
By Theorem \ref{thm-dec-UnEX}, it suffices to prove that $x\cdot L_1(x)=A_d(x)$.
We have
\begin{align*}
x\cdot L_1(x)&=\sum_{j=0}^{d!-1} x^{\left\lceil \frac{q_1(d)\cdot j}{d!}\right\rceil+\cdots
+\left\lceil \frac{q_{d-1}(d)\cdot j}{d!}\right\rceil+ \left\lceil \frac{q_d(d)\cdot j+q_d(d)}{d!}\right\rceil}
 = \sum_{j=0}^{d!-1} x^{j+1+\sum_{i=1}^{d-1} \left\lceil \frac{-j}{i!+(i-1)!}\right\rceil}
\\ & = \sum_{j=0}^{d!-1} x^{j+1-\sum_{i=0}^{d-2} \left\lfloor \frac{j}{i!(i+2)}\right\rfloor}.
\end{align*}
By Corollary \ref{Cor-Adx=descent}, this completes the proof.
\end{proof}

Note that $\mathcal{S}_d(m)$ is a special case of $\Delta(0,q^{(m)})$-simplices and is itself a member of the Eulerian simplex families.
There exist other Eulerian simplex families.
For example,
$$\widehat{\mathcal{S}}_5(m) = \mathrm{conv}\{\mathbf{0}^5, \mathbf{e}_1^5, \mathbf{e}_2^5, \mathbf{e}_3^5, \mathbf{e}_4^5, (-20,-10,3,4,5!\cdot m)\}.$$
Its Ehrhart series is given by
$$\mathrm{Ehr}(\widehat{\mathcal{S}}_5(m),x)=\frac{m(x^5+26x^4+66x^3+26x^2+x)+(-x^5-10x^4-3x^3+10x^2+3x+1)}{(1-x)^6}.$$

\begin{prop}\label{Cor-Euler-h-polynomial}
Given a positive integer $d$, if the polytope $\widetilde{\mathcal{S}}_d(m)$ fall into the Eulerian simplex family $\mathcal{R}_d(m)$ of dimension $d$, then we have
$$i(\widetilde{\mathcal{S}}_d(m),t)=m t^d  + c_{d-1} t^{d-1}+\cdots+c_{1} t+1,$$
where $c_{d-1},\ldots,c_1 $ are constants independent of $m$.
\end{prop}
\begin{proof}
We assume that the $h^*$-polynomial of $\widetilde{\mathcal{S}}_d(m)$ can be expressed as
$$h^*_{\widetilde{\mathcal{S}}_d(m)}(x) = m A_d(x) + T(x)\quad \text{with} \quad T(x)=\sum_{i=0}^{d} r_{i} x^{i}.$$
By Equation~\eqref{formual-ehrbyh*}, Proposition \ref{Generating Function for Eulerian Polynomials}, we have
$$i(\widetilde{\mathcal{S}}_d(m),t) = mt^d+\sum_{i=0}^{d} r_i \binom{t + d - i}{d}.$$
Since $T(1)=0$ (from Proposition \ref{Prop-L1L2-pro}), the coefficient of $t^d$ in the polynomial $\sum_{i=0}^{d} r_i \binom{t + d - i}{d}$ is $0$. As $T(x)$ is independent of $m$, it follows that each $r_i$ is also independent of $m$.
\end{proof}

\begin{cor}
Follow the notation above, and the following equation holds:
$$\#(\widetilde{\mathcal{S}}_d(m) \cap \mathbb{Z}^d) =  [t]T(x)+m +d+1 \quad \text{and} \quad
\#(\mathrm{int}(\widetilde{\mathcal{S}}_d(m)) \cap \mathbb{Z}^d) =  [t^d]T(x)+m.$$
\end{cor}
\begin{proof}
The conclusion follows by Equation~\eqref{formula-h^*-prop}.
\end{proof}

We now give the Ehrhart series and Ehrhart polynomials for the family $\mathcal{S}_d(m)$ of polytopes in Theorem \ref{Prop-Sk}.
\begin{cor}\label{Ehrhart-Series-of-Sdm}
Let $m\geq 1$.
For any positive integer $d$, we have
$$\mathrm{Ehr}(\mathcal{S}_d(m),x)=\frac{A_d(x)(mx-x+1)}{x(1-x)^{d+1}}.$$
\end{cor}
\begin{proof}
By Theorem \ref{Prop-Sk}, we know $x\cdot L_1(x)=A_d(x)$.
By Theorem \ref{thm-dec-UnEX}, we obtain
\begin{align*}
L_2(x)&=\sum_{j=0}^{d!-1}x^{j-\sum_{i=0}^{d-2} \left\lfloor \frac{j}{i!(i+2)}\right\rfloor}- \sum_{j=0}^{d!-1}x^{j+1-\sum_{i=0}^{d-2} \left\lfloor \frac{j}{i!(i+2)}\right\rfloor}
\\ &= A_d(x)\cdot \frac{1-x}{x}.
\end{align*}
Hence, the conclusion holds.
\end{proof}

The following corollary implies that Theorem \ref{Theorem-Introd-Eulor-Ehrhar} holds.

\begin{cor}\label{Corollary-Sdm-Ehrhart-polynomial}
Let $m\geq 1$.
For every positive integer $d$, the Ehrhart polynomials of $\mathcal{S}_d(m)$ are given by
$$i(\mathcal{S}_d(m),t)=mt^d+\sum_{i=0}^{d-1}\binom{d}{i} t^i.$$
In particular, $\mathcal{S}_d(m)$ is Ehrhart positive (i.e., all coefficients are nonnegative).
\end{cor}
\begin{proof}
By Corollary \ref{Ehrhart-Series-of-Sdm} and Proposition \ref{Generating Function for Eulerian Polynomials}, we obtain
$$\mathrm{Ehr}(\mathcal{S}_d(m),x)=\frac{(m-1) A_d(x)}{(1-x)^{d+1}}+\frac{A_d(x)}{x(1-x)^{d+1}}
=\sum_{t\geq 0}(m-1)t^dx^t+ \sum_{t\geq 0}(t+1)^d x^t.$$
This completes the proof.
\end{proof}

\begin{exa}
We present the Ehrhart series and Ehrhart polynomials for the first few families of simplices in Theorem \ref{Prop-Sk}:
\begin{small}
\begin{align*}
\mathcal{S}_1(m) & = \ell_{m},\\
\mathcal{S}_2(m) & = \mathrm{conv}\{\mathbf{0}^2, \mathbf{e}_1^2, (-1,2!\cdot m)\},  \\
\mathcal{S}_3(m) & = \mathrm{conv}\{\mathbf{0}^3, \mathbf{e}_1^3, \mathbf{e}_2^3, (-3,-2,3!\cdot m)\},
\\ \mathcal{S}_4(m) & = \mathrm{conv}\{\mathbf{0}^4, \mathbf{e}_1^4, \mathbf{e}_2^4,\mathbf{e}_3^4, (-12,-8,-3,4!\cdot m)\},
\\ \mathcal{S}_5(m) & = \mathrm{conv}\{\mathbf{0}^5, \mathbf{e}_1^5, \mathbf{e}_2^5, \mathbf{e}_3^5, \mathbf{e}_4^5, (-60,-40,-15,-4,5!\cdot m)\}.
\end{align*}
\end{small}
Their Ehrhart series are as follows:
\begin{align*}
\mathrm{Ehr}(\mathcal{S}_1(m),x)&=\frac{mx+(-x+1)}{\left(1-x \right)^{2}}, \\
 \mathrm{Ehr}(\mathcal{S}_2(m),x)&=\frac{m(x^{2}+ x) +(-x^{2}+1)}{\left(1-x \right)^{3}}, \\
\mathrm{Ehr}(\mathcal{S}_3(m),x)&=\frac{m(x^3+4x^2+x)+(-x^3-3x^2+3x+1)}{(1-x)^4}, \\
\mathrm{Ehr}(\mathcal{S}_4(m),x)&=\frac{m(x^4+11x^3+11x^2+x)+(-x^4-10x^3+10x+1)}{(1-x)^5}, \\
\mathrm{Ehr}(\mathcal{S}_5(m),x)&=\frac{m(x^5+26x^4+66x^3+26x^2+x)+(-x^5-25x^4-40x^3+40x^2+25x+1)}{(1-x)^6}.
\end{align*}
\end{exa}

\subsection{Reverse isoperimetric phenomenon}\label{sub-Section-RIP}

Given a geometric object $E\subset \mathbb{R}^d$, its \emph{volume} is defined by the integral $|E|:=\int_E d\mathbf{x}$.

The isoperimetric problem stands as one of the most classical and renowned problems in the history of geometry.
In two dimensions, the isoperimetric inequality asserts that a disk has the smallest boundary length among all domains in the plane with a given area.

The isoperimetric problem in $\mathbb{R}^d$ is to minimize the surface area among all domains having given volume, or equivalently, maximize the volume among all domains whose boundary surfaces have fixed $(d-1)$-dimensional area.
The solution in both cases is that the unique extremal is the domain bounded by a sphere.

\begin{prop}{\em \cite[Isoperimetric inrequality]{Isoperim-inequality}}\label{Sommth-Osoperimet-inequalit}
Let $E$ be a compact domain in $\mathbb{R}^d$ with smooth boundary. Then
$$|\partial E|^d\geq d^d|B_1^d|\cdot |E|^{d-1},$$
where $|E|$ denotes the volume of $E$ and $|\partial E|$ denotes the $(d-1)$-dimensional measure of the boundary $\partial E$. Moreover, $B_1^d=\{ x\in \mathbb{R}^d: |x|<1\}$ denotes the open unit ball in $\mathbb{R}^d$ and $|B_1^d|$ denotes its volume.
\end{prop}

The isoperimetric inequality is sharp on balls. Proposition \ref{Sommth-Osoperimet-inequalit} can be understood intuitively as the volume of $E$ being bounded by the volume of its boundary.
We refer the reader to the many review books and papers \cite{Chavel,Fusco,FuscoMaggi,Osserman}
Note that the isoperimetric problem above is treated in the ``continuous" context.
The discrete isoperimetric inequality has also been extensively studied; see, e.g, \cite{Hamamuki,ChungF,Bezrukov,BlockHD}.

For the lattice polytopes in our study, the isoperimetric inequality mentioned above fails to hold for their volume and boundary volume. The Reeve tetrahedron (Example \ref{ReeveTetrahedron}) provides a counterexample in dimension three, with volume $\frac{m}{6}$ and boundary volume $2$. As $m$ tends to infinity, the volume of the Reeve tetrahedron is much larger than its boundary volume.

In fact, for each dimension, by Corollary~\ref{Corollary-Sdm-Ehrhart-polynomial}, the isoperimetric inequality fails to hold.
\begin{cor}\label{Corollary-m-infinity}
Let $m\geq 1$.
For every positive integer $d$, there exists a family of $d$-dimensional integral polytopes $\mathcal{P}_d(m)$ such that, as $m$ increases, their volume tends to infinity while their boundary volume remains constant.
\end{cor}
\begin{proof}
For an integral polytope $\mathcal{P}$, it is well-known \cite[Corollary 3.20; Theorem 5.6]{BeckRobins} that the leading coefficient of $i(\mathcal{P},t)$ equals the volume of $\mathcal{P}$, the second highest coefficient of $i(\mathcal{P},t)$ equals half of the boundary volume of $\mathcal{P}$.
This result now follows directly from Theorem \ref{Prop-Sk} and Proposition \ref{Cor-Euler-h-polynomial}.
\end{proof}

Inequalities among the coefficients of Ehrhart polynomials have also been extensively studied; see, e.g., \cite{Beck-Deloera-Stanley,Betke-McMullen,Hibi1995,Stanleyh-polynomial,Stanley-Cohen-Macau}.
By Corollary \ref{Corollary-Sdm-Ehrhart-polynomial}, the following result follows immediately.

\begin{cor}
Consider any $d$-dimensional integral polytope $\mathcal{P}$ with Ehrhart polynomial:
$$i(\mathcal{P},t)=c_{\mathcal{P},d}t^d+c_{\mathcal{P},d-1}t^{d-1}+\cdots +c_{\mathcal{P},1}t+1.$$
There does not exist a function $F(y_1,y_2,\ldots,y_{d-1}): \mathbb{Q}^{d-1}\rightarrow \mathbb{R}^+$ such that for every integral polytope $\mathcal{P}$, the following inequality holds:
$$c_{\mathcal{P},d}<F(c_{\mathcal{P},1}, c_{\mathcal{P},2}, \ldots, c_{\mathcal{P},d-1}).$$
\end{cor}

We now end this subsection with a property of the $h^*$-polynomial of $\mathcal{S}_d(m)$.

A sequence of positive integers $(a_0,a_1,\ldots,a_d)$ is \emph{strictly log concave} if $a_i^2>a_{i-1}a_{i+1}$ for $1\leq i\leq d-1$ and is \emph{strictly unimodal} if $a_0<a_1<\cdots <a_i$ and $a_{i+1}>a_{i+2}>\cdots >a_d$ for some $0\leq i\leq d$.
It is well known that if $(a_0,a_1,\ldots,a_d)$ is strictly log concave, then it is strictly unimodal.
If the polynomial $a_0+a_1t+\cdots +a_dt^d$ has negative real roots, then the sequence $(a_0,a_1,\ldots,a_d)$ is strictly log concave and hence strictly unimodal \cite{Beck-Stapledon}.
An introduction to log concave and unimodal sequences can be found in \cite{StanleyLog-concave}.

The following conclusion was obtained by Beck and Stapledon in their study of properties related to the Hecke operator $U_n$.
\begin{prop}{\em \cite[Corollary 1.3]{Beck-Stapledon}}
Fix a positive integer $d$ and let $\rho_1<\rho_2<\cdots <\rho_d=0$ denote the roots of the Eulerian polynomial $A_d(x)$.
There exist positive integers $m_d$ and $n_d$ such that, if $\mathcal{P}$ is a $d$-dimensional lattice polytope and $n>n_d$, then $$h^*_{n\mathcal{P}}(x)=h_0^*(n)+h_1^*(n)x+\cdots +h_d^*(n)x^d$$
has negative real roots $\beta_1(n)<\beta_2(n)<\cdots <\beta_{d-1}(n) <\beta_{d}(n)<0$ with $\beta_i(n)\to\rho_i$ as $n\to \infty$, and the coefficients of $h^*_{n\mathcal{P}}(x)$ are positive, strictly log concave, and satisfy
$$1=h^*_0(n)<h^*_d(n)<h^*_1(n)<\cdots < h^*_i(n)<h^*_{d-i}(n)<h^*_{i+1}(n)<\cdots <h^*_{\lfloor \frac{d+1}{2}\rfloor}(n)<m_dh^*_d(n).$$
\end{prop}

Polytopes $\mathcal{S}_d(m)$ lead to a similar but distinct result, as follows.

\begin{prop}
There exist positive integers $m_d$ and $n_d$ such that, if $n>n_d$, then
$$h^*_{\mathcal{S}_d(n)}(x)=h_0^*(n)+h_1^*(n)x+\cdots +h_d^*(n)x^d$$
has negative real roots $\beta_1(n)=\rho_1<\beta_2(n)=\rho_2<\cdots <\beta_{d-1}(n)=\rho_{d-1} <\beta_{d}(n)=\frac{-1}{n-1}<0$ with $\beta_d(n)\to\rho_d=0$ as $n\to \infty$, and the coefficients of $h^*_{\mathcal{S}_d(n)}(x)$ are positive, strictly log concave, and satisfy
$$1=h^*_0(n)<h^*_d(n)<h^*_1(n)<\cdots < h^*_i(n)<h^*_{d-i}(n)<h^*_{i+1}(n)<\cdots <h^*_{\lfloor \frac{d+1}{2}\rfloor}(n)<m_dh^*_d(n).$$
\end{prop}
\begin{proof}
By Corollary \ref{Ehrhart-Series-of-Sdm}, the $h^*$-polynomial of $\mathcal{S}_d(n)$ is given by
$$h^*_{\mathcal{S}_d(n)}(x)=\frac{A_d(x)\cdot (nx-x+1)}{x}=1+(n-1)x^d+\sum_{i=1}^{d-1}((n-1)A(d,i)+A(d,i+1)) x^i.$$
Therefore, we only need to verify the last inequality chain in the above proposition.
It is well known that the Eulerian polynomials are symmetric (i.e., $A(d,i)=A(d,d+1-i)$) and strictly unimodal.
When $d$ is even, for $0\leq i< \left\lfloor \frac{d+1}{2}\right\rfloor$, we have
\begin{align*}
h^*_{d-i}(n)-h^*_{i}(n)&=(n-1)A(d,d-i)+A(d,d-i+1)-(n-1)A(d,i)-A(d,i+1)
\\ &=(n-2)(A(d,i+1)-A(d,i))>0
\end{align*}
and
\begin{align*}
h^*_{i+1}(n)-h^*_{d-i}(n)&=(n-1)A(d,i+1)+A(d,i+2)-(n-1)A(d,d-i)-A(d,d-i+1)
\\ &=A(d,i+2)-A(d,i)>0.
\end{align*}
An analogous result holds when $d$ is odd.
This completes the proof.
\end{proof}

\section{Concluding Remarks}\label{Section-Seven-CR}

One of the important contributions of this paper is the resolution of the well-known sign pattern problem of coefficients in Ehrhart polynomials. This work is based on the authors's previous paper \cite{LiuTaoXin}, particularly on the restated Theorem \ref{thm-eq-Only-Verified},
which allows us to only consider certain specific sign patterns.

In \cite{LiuTaoXin}, we resolve the Ehrhart coefficient sign pattern problem for dimensions $d\leq 9$. After constructing several simplices $\mathcal{S}_d(m)$ for small $d$, we are able to construct the integral polytopes corresponding to all sign patterns of dimension $10\leq d\leq 17$.
The relevant program is available at the following address:
\url{https://github.com/TygerLiu/TygerLiu.github.io/tree/main/Procedure/Sign-pattern-problem}.

After constructing all $\mathcal{S}_d(m)$, it is still unclear how to settle the open problem: there are too many choices $r_i$ and $m_i$, and the pattern seems chaotic. The philosophy is ``the greater the chaos, the greater the opportunity". Our proof seems natural by hindsight: the greedy decomposition seems to be the
most controllable choice.

For the reader's convenience, we outline the explicit construction of a polytope with any given sign pattern:
$$\operatorname{\mathbf{Sgn}}=(s_{d-2},s_{d-3},\ldots,s_1),\quad \text{where}\quad s_i\in \{+1,-1\}\quad \text{for all}\quad i, \quad d\geq 6.$$
We now construct an integral polytope $\mathcal{P}$ of dimension $d$ such that the Ehrhart polynomial $i(\mathcal{P},t)=1+\sum_{i=1}^{d}c_it^i$ satisfies $\operatorname{sgn}(c_{i})=s_i$ for all $1\leq i\leq d-2$.

Case 1: If $s_{d-2}=+1$, then we construct an integral polytope $\mathcal{Q}$ of dimension $d-1$ satisfying
$\operatorname{Sgn}(\mathcal{Q})=(s_{d-3},\ldots,s_1)$. Then by \cite[Embedding Theorem I]{LiuTaoXin}, $\mathcal{P}=r\mathcal{Q}\times [0,1]$ is the desired polytope
when $r$ is sufficiently large.

Case 2: If $s_{1}=+1$, then we construct an integral polytope $\mathcal{Q}$ of dimension $d-1$ satisfying
$\operatorname{Sgn}(\mathcal{Q})=(s_{d-2},\ldots,s_2)$. Then by \cite[Embedding Theorem I]{LiuTaoXin}, $\mathcal{P}=\mathcal{Q}\times [0,m]$
 is the desired polytope
when $m$ is sufficiently large.

Case 3: If $s_{d-2}=s_{d-3}=-1$ and $s_1=-1$, then we construct an integral polytope $\mathcal{Q}$ of dimension $d-3$ satisfying
$\operatorname{Sgn}(\mathcal{Q})=(-s_{d-4},\ldots,-s_2)$. Then by \cite[Embedding Theorem IV]{LiuTaoXin}, $\mathcal{P}=r\mathcal{Q}\times \mathcal{T}_m$
 is the desired polytope when $m \gg r \gg 0$.

Case 4: If $s_{1}=-1$,$s_{2}=+1$, and $s_3=-1$, then we construct an integral polytope $\mathcal{Q}$ of dimension $d-2$ satisfying
$\operatorname{Sgn}(\mathcal{Q})=(s_{d-2},\ldots,s_3)$. Then by \cite[Embedding Theorem III]{LiuTaoXin}, $\mathcal{P}=r\mathcal{Q}\times \mathrm{conv}\{(0,0), (1,0), (1,a), (2,a)\}$ is the desired polytope when $a \gg r \gg 0$.

Case 5: If $\operatorname{\mathbf{Sgn}}$ contains two consecutive $+1$, then we treat the following two cases separately.
Suppose $d=d_1+d_2$ with $d_1\geq d_2\geq 2$.

1.) When $s_{d_1}=s_{d_1-1}=+1$, we construct two integral polytopes $\mathcal{Q}_1$ (of dimension $d_1$) and $\mathcal{Q}_2$ (of dimension $d_2$) satisfying $\operatorname{Sgn}(\mathcal{Q}_1)=(s_{d_1-2},\ldots,s_1)$ and $\operatorname{Sgn}(\mathcal{Q}_2)=(s_{d-2},\ldots,s_{d_1+1})$, respectively. Then by \cite[Embedding Theorem II]{LiuTaoXin}, $\mathcal{P}=r\mathcal{Q}_1\times \mathcal{Q}_2$ is the desired polytope when $r$ is sufficiently large.

2.) When $s_{d_2}=s_{d_2-1}=+1$, we construct two integral polytopes $\mathcal{Q}_1$ (of dimension $d_1$) and $\mathcal{Q}_2$ (of dimension $d_2$) satisfying $\operatorname{Sgn}(\mathcal{Q}_1)=(s_{d-2},\ldots,s_{d_2+1})$ and $\operatorname{Sgn}(\mathcal{Q}_2)=(s_{d_2-2},\ldots,s_{1})$, respectively. Then by \cite[Embedding Theorem II]{LiuTaoXin}, $\mathcal{P}=\mathcal{Q}_1\times r\mathcal{Q}_2$ is the desired polytope when $r$ is sufficiently large.

Case 6: If $\operatorname{\mathbf{Sgn}}$ does not satisfy all of Cases 1--5, then the sign pattern reduces to the case described in Theorem \ref{thm-eq-Only-Verified}. Thus, we set
$$\mathcal{P}=\left(\prod_{i=1}^k r_i \mathcal{S}_{d_i}(m_i) \right)\times r_0\mathcal{T}_{m_0}.$$
This is the main result obtained in this paper; see Theorem \ref{Thm-main-Sgn-PP0}.

In future work, we will further explore the applications of the family of simplices $\mathcal{S}_d(m)$ in Ehrhart theory.





\noindent
{\small \textbf{Acknowledgments:}}
This work was partially supported by the National Natural Science Foundation of China [12571355].

\end{document}